\documentclass{amsart}
\usepackage{amsfonts,amssymb,amsmath,amsthm}
\usepackage{url}
\usepackage{enumerate}
\usepackage[english]{babel}
\usepackage{latexsym}  
\usepackage{longtable} 
\usepackage{array} 
\usepackage{setspace}
\usepackage{fancyhdr}
\usepackage{epsfig}
\usepackage{lscape}
\usepackage{pdflscape}
\usepackage{datetime}
\usepackage{ifpdf}
\usepackage{textcomp}
\usepackage{setspace}
\usepackage{eufrak}
\usepackage{euscript}

\usepackage{pslatex}          
\usepackage[pdftex,breaklinks]{hyperref} 

\newtheorem{theorem}{Theorem}[section]
\newtheorem{lemma}[theorem]{Lemma}
\newtheorem{proposition}[theorem]{Proposition}
\newtheorem{corollary}[theorem]{Corollary}

\theoremstyle{definition}
\newtheorem{definition}[theorem]{Definition}
\newtheorem{remark}[theorem]{Remark}
\numberwithin{equation}{section}

\newcommand{\legendre}[2]{\genfrac{(}{)}{}{}{#1}{#2}}

\hypersetup{
pdfinfo={
   Author={Eb\'{e}n\'{e}zer Ntienjem},
   Title={Elementary Evaluation of the Convolution Sum involving primitive 
Dirichlet Characters for a Class of positive Integers},
   CreationDate={D:20160209102000},
   ModDate={D:\pdfdate)=},
   Subject={2010 Mathematics Subject Classification: 11A25, 11E20, 11E25, 11F11, 11F20,  11F27},
   Keywords={Sums of Divisors; Dedekind eta function; Convolution Sums; 
Dirichlet Characters; Modular Forms; Eisenstein forms; Cusp Forms; Octonary quadratic Forms; 
Number of Representations}
}
}

\author{Eb\'{e}n\'{e}zer Ntienjem}
\address{
Centre for Research in Algebra and Number Theory \\
School of Mathematics and Statistics\\
Carleton University\\
1125 Colonel By Drive\\
Ottawa, Ontario, K1S 5B6, Canada}
\email{ebenezer.ntienjem@carleton.ca;ntienjem@gmail.com}



\def \eN{\mbox{E.~Ntienjem}}
\def \waS{\mbox{W.~A.~Stein}}
\def \mN{\mbox{M.~Newman}}
\def \gL{\mbox{G.~Ligozat}}
\def \gK{\mbox{G.~K\"{o}hler}}
\def \ljpK{\mbox{L.~J.~P.~Kilford}}
\def \aA{\mbox{A.~Alaca}}
\def \sA{\mbox{\c{S}.~Alaca}}
\def \aP{\mbox{A.~Pizer}}
\def \ksW{\mbox{K.~S.~Williams}}
\def \jgH{\mbox{J.~G.~Huard}}
\def \gaL{\mbox{G.~A.~Lomadze}}
\def \mL{\mbox{M.~Lemire}}
\def \sC{\mbox{S.~Cooper}}
\def \pcT{\mbox{P.~C.~Toh}}
\def \bR{\mbox{B.~Ramakrishnan}}
\def \bS{\mbox{B.~Sahu}}
\def \mB{\mbox{M.~Besge}}
\def \dY{\mbox{D.~Ye}}
\def \jwlG{\mbox{J.~W.~L.~Glaisher}}
\def \sR{\mbox{S.~Ramanujan}}
\def \yK{\mbox{Y.~Kesicio$\check{g}$lu}}
\def \eR{\mbox{E.~Royer}}
\def \zmO{\mbox{Z.~M.~Ou}}
\def \bkS{\mbox{B.~K.~Spearman}}
\def \hhC{\mbox{H.~H.~Chan}}
\def \exwX{\mbox{E.~X.~W.~Xia}}
\def \xlT{\mbox{X.~L.~Tian}}
\def \oxmY{\mbox{O.~X.~M.~Yao}}
\def \nK{\mbox{N.~Koblitz}}
\def \tM{\mbox{T.~Miyake}}

\def \E{\mbox{$\EuFrak{E}$}}
\def \S{\mbox{$\EuFrak{S}$}}
\def \M{\mbox{$\EuFrak{M}$}}
\def \N{\mbox{$\EuFrak{N}$}}

\begin{document}


\title[Elementary Evaluation of Convolution Sums for a Class of positive Integers, II]
{Elementary Evaluation of Convolution Sums involving primitive 
Dirichlet Characters for a Class of positive Integers
}

\begin{abstract}
We extend the results obtained by \eN\ \cite{ntienjem2016c} to all 
positive integers. 
Let $\N$ be the subset of $\mathbb{N}$ consisting of $\,2^{\nu}\mho$, where $\nu$ is in $\{0,1,2,3\}$ and $\mho$ 
is a squarefree finite product of distinct odd primes.  
We discuss the evaluation of the convolution sum, 
$\underset{\substack{ {(l,m)\in\mathbb{N}^{2}} \\ {\alpha\,l+\beta\,m=n}
} }{\sum}\sigma(l)\sigma(m)$, when $\alpha\beta$ is in  
$\mathbb{N}\setminus\N$. 
The evaluation of convolution sums belonging to this class is achieved by applying 
modular forms and primitive Dirichlet characters. 
In addition, we revisit the evaluation of the convolution sums for $\alpha\beta=9$, 
$16$, $18$, $25$, $36$.  
If $\alpha\beta\equiv 0 \pmod{4}$, we determine 
natural numbers $a,b$ and use the evaluated convolution sums together
with other known convolution sums to carry out 
the number of representations of $n$ by the octonary quadratic forms
$a\,(x_{1}^{2} + x_{2}^{2} + x_{3}^{2} + x_{4}^{2})+ b\,(x_{5}^{2} + x_{6}^{2}
  + x_{7}^{2} + x_{8}^{2})$.
Similarly, if $\alpha\beta\equiv 0 \pmod{3}$, we compute natural numbers $c,d$ 
and make use of the 
evaluated convolution sums together with other known convolution sums 
to determine the number of representations of $n$ by the octonary quadratic forms
$c\,(\,x_{1}^{2} + x_{1}x_{2} + x_{2}^{2} + x_{3}^{2} + x_{3}x_{4} + x_{4}^{2}\,)
+ d\,(\,x_{5}^{2} + x_{5}x_{6} + x_{6}^{2} + x_{7}^{2} + x_{7}x_{8} + x_{8}^{2}\,)$. 
We illustrate our method with the explicit examples  
$\alpha\beta = 3^{2}\cdot 5$,  $\alpha\beta = 2^{4}\cdot 3$,  
$\alpha\beta = 2\cdot 5^{2}$ and $\alpha\beta = 2^{6}$, . 
\end{abstract}

\maketitle

\section{Introduction} \label{introduction}

In this work, we denote by $\mathbb{N}$, $\mathbb{Z}$, $\mathbb{Q}$, 
$\mathbb{R}$ 
and $\mathbb{C}$ the sets of natural numbers, integers, rational 
numbers, real numbers and complex numbers, respectively. 
Let $i,j,k,l,m,n\in\mathbb{N}$ in the sequel. The sum of positive divisors 
of $n$ to the power of $k$, $\sigma_{k}(n)$, is defined by 
\begin{equation} \label{def-sigma_k-n}
 \sigma_{k}(n) = \sum_{0<d|n}d^{k}.
 \end{equation} 
We let $\sigma(n)$ stand as a synonym for $\sigma_{1}(n)$. 
For $m\notin\mathbb{N}$ we set $\sigma_{k}(m)=0$. 

Let $\alpha,\beta\in\mathbb{N}$ be such that $\alpha\leq\beta$.
We define the convolution sum, 
 $W_{(\alpha,\beta)}(n)$, as follows:  
 \begin{equation} \label{def-convolution_sum}
 W_{(\alpha, \beta)}(n) =  \sum_{\substack{
 {(l,m)\in\mathbb{N}^{2}} \\ {\alpha\,l+\beta\,m=n}
} }\sigma(l)\sigma(m).
 \end{equation}
We write $W_{\beta}(n)$ as a synonym for $W_{(1,\beta)}(n)$. 
If for all $(l,m)\in\mathbb{N}^{2}$ it holds that 
$\alpha\,l+\beta\,m\neq n$, then we set $W_{(\alpha,\beta)}(n)=0$ 

So far known convolution sums are displayed on \autoref{introduction-table-1}.

\begin{longtable}{|r|r|r|} \hline 
\textbf{Level $\alpha\beta$}  &  \textbf{Authors}   &  \textbf{References}  \\ \hline
1  &  \mB, \jwlG, & ~ \\
 ~  & \sR\  & \cite{besge,glaisher,ramanujan} \\ \hline
2, 3, 4  & \jgH\ \& \zmO\ \&   &  ~  \\
 ~  & \bkS\ \& \ksW\   & \cite{huardetal} \\ \hline
5, 7  & \mL\ \& \ksW,  & ~  \\
 ~ & \sC\ \& \pcT\   & \cite{cooper_toh,lemire_williams} \\ \hline
6  & \sA\ \& \ksW\   & \cite{alaca_williams} \\ \hline
8, 9  & \ksW\   & \cite{williams2006, williams2005} \\ \hline
10, 11, 13, 14 &  \eR   & \cite{royer} \\ \hline
12, 16, 18, 24 &  \aA\ \& \sA\  &  ~  \\ 
 ~  & \& \ksW\   &
\cite{alaca_alaca_williams2006,alaca_alaca_williams2007,alaca_alaca_williams2007a,alaca_alaca_williams2008}
\\ \hline
15  & \bR\ \& \bS\   & \cite{ramakrishnan_sahu} \\ \hline
10, 20  & \sC\ \& \dY\   & \cite{cooper_ye2014} \\ \hline
23  & \hhC\ \& \sC\   & \cite{chan_cooper2008} \\ \hline
25  & \exwX\ \& \xlT\  &  ~  \\ 
 ~  & \& \oxmY   & \cite{xiaetal2014} \\ \hline
27, 32  & \sA\ \& \yK   & \cite{alaca_kesicioglu2014} \\ \hline
36  & \dY\   & \cite{ye2015} \\ \hline
14, 26, 28, 30 &  \eN  & \cite{ntienjem2015} \\ \hline
22, 44, 52  &  \eN\  & \cite{ntienjem2016b}  \\ \hline
33, 40, 56  & ~  &  ~ \\ 
$\alpha\beta=2^{\nu}\underset{j\geq 2}{\overset{\kappa}{\prod}}\,p_{j}$, where 
 & ~  & ~ \\
$\gcd(\alpha,\beta)=1$, $0\leq\nu\leq 3$, &   \eN\  & \cite{ntienjem2016c}  \\
$\kappa\in\mathbb{N}$, $p_{j}>2$ distinct primes  & ~   & ~ \\ \hline
\caption{Known convolution sums $W_{(\alpha, \beta)}(n)$ of level $\alpha\beta$} 
\label{introduction-table-1}
\end{longtable}

Let 
\begin{equation*}  \label{introd-eqn-N}
\N=\{\,2^{\nu}\mho\,|\,\nu\in\{0,1,2,3\} \text{ and $\mho$ is a squarefree  
finite product of distinct odd primes}\,\} 
\end{equation*}  
be a subset of $\mathbb{N}$.  

We evaluate the convolution sum, $W_{(\alpha,\beta)}(n)$, for the class of 
natural numbers $\alpha\beta$ such that $\alpha\beta\in(\mathbb{N}\setminus\N)$.
We use Dirichlet characters and modular forms to evaluate these convolution sums. 

We observe that the positive integers $\alpha\beta=9,16,18,25,27,32,36$ from 
\autoref{introduction-table-1} belong 
to the class of integers for which the evaluation of the convolution sum 
is discussed in this paper. 
From these integers, the convolution sums for $27$ and $32$ are evaluated 
using the approach that we are generalizing in the sequel. 
We revisit the evaluation of the convolution sums for $\alpha\beta=9$, 
$16$, $18$, $25$, $36$ using our method.  

We use the result from the above general case to obtain 
the evaluation of the convolution sum for 
$\alpha\beta=3^{2}\cdot 5$, $\alpha\beta = 2^{4}\cdot 3$,  $\alpha\beta=2\cdot 5^{2}$ 
and $\alpha\beta = 2^{6}$.  
These convolution sums  have not been evaluated as yet.

As an application, convolution sums are used to determine explicit 
formulae for the number of representations of a positive 
integer $n$ by the octonary quadratic forms 
\begin{equation} \label{introduction-eq-1}
a\,(x_{1}^{2} +x_{2}^{2} + x_{3}^{2} + x_{4}^{2})+ b\,(x_{5}^{2} + x_{6}^{2} + 
x_{7}^{2} + x_{8}^{2}) 
\end{equation}
and 
\begin{equation} \label{introduction-eq-2}
c\,(x_{1}^{2} + x_{1}x_{2} + x_{2}^{2} + x_{3}^{2} + x_{3}x_{4} + x_{4}^{2})
+ d\,(x_{5}^{2} + x_{5}x_{6}+ x_{6}^{2} + x_{7}^{2} + x_{7}x_{8}+ 
x_{8}^{2}),
\end{equation}
respectively, where $a, b,c, d\in \mathbb{N}$. 

So far known explicit formulae for the number of  
representations of $n$ by the octonary form 
\autoref{introduction-eq-1}  
are displayed in 
\autoref{introduction-table-rep2}. 

\begin{longtable}{|r|r|r|} \hline 
$\mathbf{(a,b)}$  &  \textbf{Authors}   &  \textbf{References}  \\ \hline
(1,1),(1,3), & ~ & ~ \\
(1,9),(2,3)  & \eN\   & \cite{ntienjem2016c} \\ \hline
(1,2)  & \ksW\   & \cite{williams2006} \\ \hline
(1,4)  & \aA\ \& \sA\  &  ~  \\ 
 ~  & \& \ksW\   & \cite{alaca_alaca_williams2007} \\ \hline
(1,5)  & \sC\ \& \dY\   & \cite{cooper_ye2014} \\ \hline
(1,6)  & \bR\ \& \bS\   & \cite{ramakrishnan_sahu}  \\ \hline
(1,7) & \eN\  & \cite{ntienjem2015} \\ \hline 
(1,8)  & \sA\ \& \yK\   & \cite{alaca_kesicioglu2014} \\ \hline
(1,11),(1,13)   &  \eN\  & \cite{ntienjem2016b}  \\ \hline
(1,10),(1,14),  &  ~  & ~ \\ 
(2,5),(2,7),  &  ~  & ~ \\ 
$a\,b=2^{\nu}\underset{j\geq 2}{\overset{\kappa}{\prod}}\,p_{j}$, where 
 & ~  & ~ \\
$\gcd(a,b)=1$, $0\leq\nu\leq 1$, & \eN\  & \cite{ntienjem2016c} \\
$\kappa\in\mathbb{N}$, $p_{j}>2$ distinct primes  & ~   & ~ \\ \hline
\caption{Known representations of $n$ by the form \autoref{introduction-eq-1}
} 
\label{introduction-table-rep2}
\end{longtable}

Similarly, so far known explicit formulae for the number of  
representations of $n$ by the octonary form \autoref{introduction-eq-2} 
are referenced in 
\autoref{introduction-table-rep1}.

\begin{longtable}{|r|r|r|} \hline 
$\mathbf{(c,d)}$  &  \textbf{Authors}   &  \textbf{References}  \\ \hline
(1,1)  &  \gaL\  & \cite{lomadze} \\ \hline
(1,2)  & \sA\ \& \ksW\   & \cite{alaca_williams}    \\ \hline
(1,3)  & \ksW\   & \cite{williams2005}    \\ \hline 
(1,4),(1,6),  & ~ & ~ \\
(1,8),(2,3)  & \aA\ \& \sA\  &  ~  \\ 
 ~  &  \& \ksW\   & \cite{alaca_alaca_williams2006,alaca_alaca_williams2007,alaca_alaca_williams2007a} \\ \hline 
(1,5)  & \bR\ \& \bS\   & \cite{ramakrishnan_sahu}  \\ \hline
(1,9)  & \sA\ \& \yK\   & \cite{alaca_kesicioglu2014} \\ \hline
(1,10), (2,5) & \eN\  & \cite{ntienjem2015} \\ \hline 
(1,12),(3,4)  & \dY\   & \cite{ye2015} \\ \hline
(1,11),  &  ~  &  ~ \\ 
$c\,d=2^{\nu}\underset{j\geq 3}{\overset{\kappa}{\prod}}\,p_{j}$, where 
 & ~  & ~ \\
$\gcd(c,d)=1$, $0\leq\nu\leq 3$, & \eN\  & \cite{ntienjem2016c} \\
$\kappa\in\mathbb{N}$, $p_{j}>3$ distinct primes  & ~   & ~ \\ \hline
\caption{Known representations of $n$ by the form \autoref{introduction-eq-2}
} 
\label{introduction-table-rep1}
\end{longtable}
 
We also determine explicit formulae for the number of representations 
of a positive integer $n$ by such octonary quadratic forms whenever  
$\alpha\beta\equiv 0\pmod{4}$ or $\alpha\beta\equiv 0\pmod{3}$. 

We then use the convolution sums,  $W_{(\alpha,\beta)}(n)$, where 
$\alpha\beta = 3^{2}\cdot 5, 2^{4}\cdot 3, 2^{6}$, to
give examples of explicit formulae for the number of representations 
of a positive 
integer $n$ by the octonary quadratic forms \autoref{introduction-eq-1} and  
\autoref{introduction-eq-2}.

This paper is organized as follows. In 
\hyperref[modularForms]{Section \ref*{modularForms}} we discuss modular forms  
and briefly define $\eta$-functions  and 
convolution sums. We assume that   
$\alpha\beta$ has the above form and then discuss the evaluation of 
the convolution sum,  $W_{(\alpha,\beta)}(n)$, in 
\hyperref[convolution_alpha_beta]{Section \ref*{convolution_alpha_beta}}. In 
\hyperref[representations_a_b]{Section \ref*{representations_a_b}} and 
\hyperref[representations_c_d]{Section \ref*{representations_c_d}}, we 
discuss a technique for computing all pairs of natural numbers $(a,b)$ 
and $(c,d)$, and then determine explicit 
formulae for the number of representations of $n$ by the octonary form 
\autoref{introduction-eq-1} and \autoref{introduction-eq-2} when 
$\alpha\beta\equiv 0\pmod{4}$ or $\alpha\beta\equiv 0\pmod{3}$. In 
\hyperref[convolution_45_50]{Section \ref*{convolution_45_50}}, we evaluate the convolution sums
$W_{(1,45)}(n)$, $W_{(5,9)}(n)$, $W_{(1,48)}(n)$, $W_{(3,16)}(n)$, $W_{(1,50)}(n)$, 
$W_{(2,25)}(n)$ and $W_{(1,64)}(n)$ ; then in 
\hyperref[representations_5_9]{Section \ref*{representations_5_9}}, 
we make use of these convolution sums and other known convolution sums 
to determine an explicit formula for the number of
representations of a positive integer $n$ by the octonary quadratic form 
\begin{itemize}
	\item \autoref{introduction-eq-1}, 
	where $(a,b)$ stands for $(1, 12)$, $(3,4)$, $(1, 16)$. 
	\item \autoref{introduction-eq-2}, where $(c,d)$ stands for $(1,15)$, 
	$(3,5)$, $(1, 16)$. 
\end{itemize}
The evaluation of the convolution sums for $\alpha\beta=9$, 
$16$, $18$, $25$, $36$ is revisited in 
\hyperref[Revisited-Evaluations]{Section \ref*{Revisited-Evaluations}}. 
Outlook and concluding remarks are made in 
\hyperref[conclusion]{Section \ref*{conclusion}}.

Software for symbolic scientific computation is used to obtain the results 
of this paper.  This software comprises the open source software packages 
\emph{GiNaC}, \emph{Maxima}, \emph{Reduce}, \emph{SAGE} and the commercial software 
package \emph{MAPLE}.


\section{Essentials to the understanding of the problem} \label{modularForms}

\subsection{Modular Forms} \label{modForms}

Let $\mathbb{H}=\{ z\in \mathbb{C}~ | ~\text{Im}(z)>0\}$, 
be the upper half-plane 
and let  $\Gamma=G=\text{SL}_{2}(\mathbb{R})= \{\,\left(\begin{smallmatrix} a
    & b \\ c & d \end{smallmatrix}\right)\,\mid\, a,b,c,d\in\mathbb{R} 
\text{ and } ad-bc=1\,\}$  
be the group of $2\times 2$-matrices. Let $N\in\mathbb{N}$. Then  
\begin{eqnarray*}
\Gamma(N) & = \bigl\{~\left(\begin{smallmatrix} a & b \\ c &
  d \end{smallmatrix}\right)\in\text{SL}_{2}(\mathbb{Z})~ |
  ~\left(\begin{smallmatrix} a & b \\ c &
 d\end{smallmatrix}\right)\equiv\left(\begin{smallmatrix} 1 & 0 \\ 0 & 1 
 \end{smallmatrix}\right) \pmod{N} ~\bigr\}
\end{eqnarray*}
is a subgroup of $G$ and is called a \emph{principal congruence subgroup of 
level N}. A subgroup $H$ of $G$ is called a \emph{congruence subgroup of 
level N} if it contains $\Gamma(N)$.

For our purposes the following congruence subgroup is relevant:
 \begin{align*}
\Gamma_{0}(N) & = \bigl\{~\left(\begin{smallmatrix} a & b \\ c &
  d \end{smallmatrix}\right)\in\text{SL}_{2}(\mathbb{Z})~ | ~
   c\equiv 0 \pmod{N} ~\bigr\}. 
\end{align*}
Let $k,N\in\mathbb{N}$ and let $\Gamma'\subseteq\Gamma$ be a congruence 
subgroup of level $N\in\mathbb{N}$. 
Let $k\in\mathbb{Z}, \gamma\in\text{SL}_{2}(\mathbb{Z})$ and $f : 
\mathbb{H}\cup\mathbb{Q}\cup\{\infty\} \rightarrow 
\mathbb{C}\cup\{\infty\}$. 
We denote by 
$f^{[\gamma]_{k}}$ the function whose value at $z$ is $(cz+d)^{-
k}f(\gamma(z))$, i.e., $f^{[\gamma]_{k}}(z)=(cz+d)^{-k}f(\gamma(z))$. 
The following definition is according to \nK\ \cite[p.\ 108]{koblitz-1993}.
\begin{definition} \label{modularForms-defin-2}
Let $N\in\mathbb{N}$, $k\in\mathbb{Z}$, $f$ be a meromorphic function 
on $\mathbb{H}$ and $\Gamma'\subset\Gamma$ a congruence subgroup of 
level $N$. 
\begin{enumerate}
\item[(a)] $f$ is called a \emph{modular function of weight $k$} for 
$\Gamma'$ if
\begin{enumerate}
\item[(a1)] for all $\gamma\in\Gamma'$ it holds that $f^{[\gamma]_{k}}=f$.
\item[(a2)] for any $\delta\in\Gamma$ it holds that $f^{[\delta]_{k}}(z)$ 
has the form 
$\underset{n\in\mathbb{Z}}{\sum}a_{n}e^{\frac{2\pi i z n}{N}}$ 
and $a_{n}\neq 0$ for finitely many $n\in\mathbb{Z}$ such that $n<0$. 
\end{enumerate}
\item[(b)] $f$ is called a \emph{modular form of weight $k$} for 
$\Gamma'$ if
	\begin{enumerate}
	\item[(b1)] $f$ is a modular function of weight $k$ for $\Gamma'$,
	\item[(b2)] $f$ is holomorphic on $\mathbb{H}$,
	\item[(b3)] for all $\delta\in\Gamma$ and for all $n\in\mathbb{Z}$ 
such that $n<0$ it holds that $a_{n}=0$.
	\end{enumerate}
\item[(c)] $f$ is called a \emph{cusp form of weight $k$ for $\Gamma'$} if
	\begin{enumerate}
	\item[(c1)] $f$ is a modular form of weight $k$ for $\Gamma'$, 
	\item[(c2)] for all $\delta\in\Gamma$ it holds that $a_{0}=0$.
	\end{enumerate}
\end{enumerate}
\end{definition}
Let us denote by $\M_{k}(\Gamma')$ the set of modular forms of weight $k$ 
for $\Gamma'$,  by $\S_{k}(\Gamma')$ the set of cusp forms of 
weight $k$ for $\Gamma'$ and by $\E_{k}(\Gamma')$ the set of Eisenstein forms. 
The sets $\M_{k}(\Gamma'),\,\S_{k}(\Gamma')$ and $\E_{k}(\Gamma')$ 
are vector spaces over $\mathbb{C}$. 
Therefore, $\M_{k}(\Gamma_{0}(N))$ is the space of 
modular forms of weight $k$ for 
$\Gamma_{0}(N)$, $\S_{k}(\Gamma_{0}(N))$ is the space of 
cusp forms of weight $k$ for $\Gamma_{0}(N)$, 
and $\E_{k}(\Gamma_{0}(N))$ is the space of Eisenstein forms. 
Consequently, \waS\ \cite[p.~81]{wstein} has shown that 
$\M_{k}(\Gamma_{0}(N))=\E_{k}(\Gamma_{0}(N))\oplus\S_{k}(\Gamma_{0}(N))$.

We asume in this paper that $4\leq k$ is even and that $\chi$ and
$\psi$ are primitive Dirichlet characters with conductors $L$ and $R$,
respectively. \waS\ \cite[p.~86]{wstein} has noted that  
\begin{equation} \label{Eisenstein-gen}
E_{k,\chi,\psi}(q) = C_{0} + \underset{n=1}{\overset{\infty}{\sum}}\,\biggl(
\underset{d|n}{\sum}\,\psi(d)\chi(\frac{n}{d})\,d^{k-1}\,\biggr)q^{n}, 
\end{equation}
where 
\begin{equation*}
C_{0} = \begin{cases}
          0 & \text{ if } L >1 \\
          -\frac{B_{k,\chi}}{2k} & \text{ if } L=1
       \end{cases}
\end{equation*}
and $B_{k,\chi}$ are the generalized Bernoulli numbers. 
Theorems 5.8 and 5.9 in Section 5.3 of \waS\ \cite[p.~86]{wstein} are 
then applicable.

\subsection{$\eta$-Quotients}  \label{etaFunctions}

On the upper half-plane $\mathbb{H}$, the Dedekind $\eta$-function, $\eta(z)$,  
is defined by
$\eta(z) = e^{\frac{2\pi i z}{24}}\overset{\infty}{\underset{n=1}{\prod}}(1 - e^{2\pi i n z})$.
Let us set $q=e^{2\pi i z}$. Then it follows that 
\begin{equation*}
\eta(z) = q^{\frac{1}{24}}\overset{\infty}{\underset{n=1}{\prod}}(1 - q^{n}) = q^{\frac{1}{24}} F(q),\quad 
\text{where } F(q)=\overset{\infty}{\underset{n=1}{\prod}} (1 - q^{n}).
\end{equation*}
We will use eta function, eta quotient and eta product interchangeably 
as synonyms.

\mN\ \cite{newman_1957,newman_1959} applied the Dedekind $\eta$-function 
to systematically construct modular forms for $\Gamma_{0}(N)$. 
Newman then 
etablishes conditions (i)-(iv) in 
the following theorem. \gL\ \cite{ligozat_1975} determined the order 
of vanishing of an $\eta$-function at all cusps of $\Gamma_{0}(N)$, which 
is condition (v) or (v$'$) in the following theorem.

\ljpK\ \cite[p.~99]{kilford}  and \gK\ \cite[p.~37]{koehler} have 
formulated the following theorem; it will be used 
to exhaustively determine $\eta$-quotients, $f(z)$, which 
belong to $\M_{k}(\Gamma_{0}(N))$, and especially those $\eta$-quotients 
which are in $\S_{k}(\Gamma_{0}(N))$. 

\begin{theorem}[\mN\ and \gL\ ] \label{ligozat_theorem} 
Let $N\in \mathbb{N}$, $D(N)$ be the set of all positive divisors of 
$N$, $\delta\in D(N)$ and $r_{\delta}\in\mathbb{Z}$. Let furthermore  
$f(z)=\overset{}{\underset{\delta\in D(N)}{\prod}}\eta^{r_{\delta}}(\delta z)$ 
be an $\eta$-quotient. 
If the following five conditions are satisfied

\begin{tabular}{llll}
{\textbf{(i)}} & $\overset{}{\underset{\delta\in D(N)}{\sum}}\delta\,r_{\delta} 
	\,\equiv 0 \pmod{24}$, & 
{\textbf{(ii)}} &  $\overset{}{\underset{\delta\in D(N)}{\sum}}\frac{N}{\delta} 
	\,r_{\delta}\, \equiv 0 \pmod{24}$, \\
{\textbf{(iii)}} &  $\overset{}{\underset{\delta\in D(N)}
{\prod}}\delta^{r_{\delta}}$ \text{ is a square in } $\mathbb{Q}$, & 
{\textbf{(iv)}} &  $0 < \overset{}{\underset{\delta\in D(N)}
	{\sum}}r_{\delta}\, \equiv 0 \pmod{4}$,  
\end{tabular}

{\textbf{(v)}} \text{ for each $d\in D(N)$ it holds that } 
	$\overset{}{\underset{\delta\in D(N)}{\sum}}\frac{\gcd{(\delta,d)}^{2}}	
	{\delta}\,r_{\delta} \geq 0$,  

then $f(z)\in\M_{k}(\Gamma_{0}(N))$, where 
$k=\frac{1}{2}\overset{}{\underset{\delta\in D(N)}{\sum}}r_{\delta}$. 

Moreover, the $\eta$-quotient $f(z)$ is an element of $\S_{k}(\Gamma_{0}(N))$ 
if ${\textbf{(v)}}$ is replaced by
 
{\textbf{(v')}} \text{ for each $d\in D(N)$ it holds that } 
$\overset{}{\underset{\delta\in D(N)}{\sum}}\frac{\gcd{(\delta,d)}^{2}}	
{\delta} r_{\delta} > 0$.
\end{theorem}

\subsection{Convolution Sums $W_{(\alpha, \beta)}(n)$}
\label{convolutionSumsEqns}

Given $\alpha,\beta\in\mathbb{N}$ such that $\alpha\leq\beta$, let 
the convolution sum be defined  by  \autoref{def-convolution_sum}.

\eN\ \cite{ntienjem2015,ntienjem2016c} has shown and \aA\ et al.\ 
\cite{alaca_alaca_williams2006} has remarked that one can simply assume 
that $\text{gcd}(\alpha,\beta)=1$.  

Let $q\in\mathbb{C}$ be such that $|q|<1$. Let furthermore $\chi$ and
$\psi$ be primitive Dirichlet characters with conductors $L$ and $R$,
respectively. We assume that $\chi=\psi$ and that $\chi$ is a Kronecker symbol 
in the following. Then the following Eisenstein series hold. 
\begin{align}  
L(q) = E_{2}(q) = 1-24\,\sum_{n=1}^{\infty}\sigma(n)q^{n}, 
\label{evalConvolClass-eqn-3} \\
M(q) = E_{4}(q) = 1+240\,\sum_{n=1}^{\infty}\sigma_{3}(n)q^{n}, 
\label{evalConvolClass-eqn-4a} \\
M_{\chi}(q) = E_{4,\chi}(q) = C_{0} + 
\sum_{n=1}^{\infty}\,\chi\,\sigma_{3}(n)\,q^{n}. 
\label{evalConvolClass-eqn-4}
\end{align}
Note that $M(q)$ is a special case of \autoref{Eisenstein-gen} or 
\autoref{evalConvolClass-eqn-4}. 
We state two relevant results for the sequel of this work.

\begin{lemma}  \label{evalConvolClass-lema-1}
Let $\alpha, \beta \in \mathbb{N}$. Then 
\begin{equation*}
( \alpha\, L(q^{\alpha}) - \beta\, L(q^{\beta}) )^{2}\in
\M_{4}(\Gamma_{0}(\alpha\beta)).
\end{equation*}
\end{lemma}
\begin{proof} 
See \eN\ \cite{ntienjem2016c}.
\end{proof}
\begin{theorem} \label{convolutionSum_a_b}
Let $\alpha,\beta,N\in\mathbb{N}$ be such that
$N=\alpha\beta, \alpha < \beta$, and $\alpha$ and $\beta$ are relatively prime. 
Then
\begin{multline}
 ( \alpha\, L(q^{\alpha}) - \beta\, L(q^{\beta}) )^{2}  = 
 (\alpha - \beta)^{2} 
    + \sum_{n=1}^{\infty}\biggl(\ 240\,\alpha^{2}\,\sigma_{3}
    (\frac{n}{\alpha}) 
    + 240\,\beta^{2}\,\sigma_{3}(\frac{n}{\beta}) \\ 
    + 48\,\alpha\,(\beta-6n)\,\sigma(\frac{n}{\alpha}) 
    + 48\,\beta\,(\alpha-6n)\,\sigma(\frac{n}{\beta}) 
    - 1152\,\alpha\beta\,W_{(\alpha,\beta)}(n)\,\biggr)q^{n}. 
    \label{evalConvolClass-eqn-11}
\end{multline}
\end{theorem}
\begin{proof} 
See \eN\ \cite{ntienjem2016c}.
\end{proof}


\section{Evaluating  $W_{(\alpha,\beta)}(n)$, where 
$\alpha\beta\in\mathbb{N}\setminus\N$}
\label{convolution_alpha_beta}

We carry out an explicit formula for the convolution sum $W_{(\alpha,\beta)}(n)$.

\subsection{Bases of $\E_{4}(\Gamma_{0}(\alpha\beta))$ and $\S_{4}(\Gamma_{0}(\alpha\beta))$}  \label{convolution_alpha_beta-bases}

Let $\EuScript{D}(\alpha\beta)$ denote the set of all positive divisors of 
$\alpha\beta$.

\aP\ \cite{pizer1976} has discussed the existence of a basis of the 
space of cusp forms of weight $k\geq 2$ for 
$\Gamma_{0}(\alpha\beta)$ when $\alpha\beta$ is not a perfect square. 
We apply the dimension formulae in \tM 's book  
\cite[Thrm 2.5.2,~p.~60]{miyake1989} or \waS 's book  
\cite[Prop.\ 6.1, p.\ 91]{wstein} to conclude that 
\begin{itemize} 
 \item  for the space of Eisenstein forms 
 \begin{equation} \label{dimension-Eisenstein}
  \text{dim}(\E_{4}(\Gamma_{0}(\alpha\beta)))=\underset{d|\alpha\beta}{\sum}
\,\varphi(\gcd(d,\frac{\alpha\beta}{d}))=m_{E},
\end{equation}
 where $m_{E}\in\mathbb{N}$	and $\varphi$ is the Euler's totient function. 
\item for the space of cusp forms  
$\text{dim}(\S_{4}(\Gamma_{0}(\alpha\beta)))=m_{S}$, where 
$m_{S}\in\mathbb{N}$. 
\end{itemize} 
We use \autoref{ligozat_theorem} $(i)-(v')$ to exhaustively determine 
as many elements of the space $\S_{4}(\Gamma_{0}(\alpha\beta))$ 
as possible. From these elements of the space 
$\S_{4}(\Gamma_{0}(\alpha\beta))$ we select relevant ones for the 
purpose of the determination of a basis of this space. 

Let $\EuScript{C}$ denote the set of Dirichlet characters 
$\chi=\legendre{m}{n}$ as assumed in \autoref{evalConvolClass-eqn-4}, 
where $m,n\in\mathbb{Z}$ and $\legendre{m}{n}$ is the Kronecker symbol. 
Let furthermore $D(\chi)\subseteq\EuScript{D}(\alpha\beta)$ denote the 
subset of $\EuScript{D}(\alpha\beta)$ associated with the character $\chi$. 

\begin{theorem} \label{basisCusp_a_b}
\begin{enumerate}
\item[\textbf{(a)}] The set $\EuScript{B}_{E}=\{\, M(q^{t})\,\mid\, t\in\EuScript{D}
(\alpha\beta)\,\}\cup\underset{\chi\in\EuScript{C}}{\bigcup}\{\, 
M_{\chi}(q^{t})\mid t\in D(\chi)\,\}$ 
is a basis of $\E_{4}(\Gamma_{0}(\alpha\beta))$.
\item[\textbf{(b)}] Let $1\leq i\leq m_{S}$ be positive integers, 
$\delta\in D(\alpha\beta)$ and $(r(i,\delta))_{i,\delta}$ be a 
table of the powers of\  $\eta(\delta\,z)$. Let furthermore  
$\EuFrak{B}_{\alpha\beta,i}(q)=\underset{\delta|\alpha\beta}{\prod}\eta^{r(i,\delta)}(\delta\,z)$ 
be selected elements of $\S_{4}(\Gamma_{0}(\alpha\beta))$. 
Then the set $\EuScript{B}_{S}=\{\, \EuFrak{B}_{\alpha\beta,i}(q)\,\mid\,
1\leq i\leq m_{S}\,\}$ is a basis of $\S_{4}(\Gamma_{0}(\alpha\beta))$. 
\item[\textbf{(c)}] The set $\EuScript{B}_{M}=\EuScript{B}_{E}\cup\EuScript{B}_{S}$ 
constitutes a basis of $\M_{4}(\Gamma_{0}(\alpha\beta))$. 
\end{enumerate}
\end{theorem}
\begin{remark} \label{basis-remark}
\begin{enumerate}
\item[(r1)] Each eta quotient $\EuFrak{B}_{\alpha\beta,i}(q)$ is expressible 
in the form 
$\underset{n=1}{\overset{\infty}{\sum}}\EuFrak{b}_{\alpha\beta,i}(n)q^{n}$, where 
for each $n\geq 1$ the coefficient $\EuFrak{b}_{\alpha\beta,i}(n)$ is an integer. 
\item[(r2)] 
When we divide the sum that results from \autoref{ligozat_theorem} $(v')$, when $d=N$, 
by $24$, then we obtain the smallest positive degree of $q$ in 
$\EuFrak{B}_{\alpha\beta,i}(q)$.
\end{enumerate}
\end{remark}

\begin{proof}
\begin{enumerate} 
\item[(a)] 
\waS\ \cite[Thrms 5.8 and 5.9, p.~86]{wstein} has shown that for each 
$t$ positive divisor of $\alpha\beta$ it holds that 
$M(q^{t})$ is in $\M_{4}(\Gamma_{0}(t))$. Since $\M_{4}(\Gamma_{0}(t))$ is a 
vector space, it also holds for each Kronecker symbol 
$\chi\in\EuScript{C}$ and $s\in D(\chi)$ that $M_{\chi}(q^{s})$ is in 
$\M_{4}(\Gamma_{0}(s))$.
Since the dimension of $\E_{4}(\Gamma_{0}(\alpha\beta))$ is finite, it suffices 
to show that $\EuScript{B}_{E}$ is linearly independent. 
Suppose that for each $\chi\in\EuScript{C},s\in D(\chi)$ we have 
$z(\chi)_{s}\in\mathbb{C}$ and  
that for each $t|\alpha\beta$ we have $x_{t}\in\mathbb{C}$.  Then 
\begin{equation*}
\underset{t|\alpha\beta}{\sum}x_{t}\,M(q^{t}) + \underset{\chi\in\EuScript{C}}
{\sum}\biggl(\,\underset{s\in D(\chi)}{\sum}z(\chi)_{s} M_{\chi}(q^{s})\,\biggr)=0. 
\end{equation*}
We recall that $\chi$ is a Kronecker symbol; therefore, for all 
$0\neq a\in\mathbb{Z}$ it holds that $\legendre{a}{0}=0$. 
Then we equate the coefficients of 
$q^{n}$ for $n\in D(\alpha\beta)\cup\underset{\chi\in\EuScript{C}}{\bigcup}
\{s| s\in D(\chi)\}$ 
to obtain the homogeneous system of linear equations in $m_{E}$ unknowns:
\begin{equation*}
\underset{u|\alpha\beta}{\sum}\,\sigma_{3}(\frac{t}{u})x_{u} + 
\underset{\chi\in\EuScript{C}}
{\sum}\,\underset{v\in D(\chi)}{\sum}\,\chi\sigma_{3}
(\frac{t}{v})Z(\chi)_{v} = 0,\qquad t\in D(\alpha\beta).
\end{equation*}
The determinant of the matrix of this homogeneous system of linear equations 
is not zero. Hence, the unique solution is 
$x_{t}=z(\chi)_{s}=0$ for all $t\in D(\alpha\beta)$ and for all 
$\chi\in\EuScript{C},s\in D(\chi)$. So, the set 
$\EuScript{B}_{E}$ is linearly independent and hence  
is a basis of $\E_{4}(\Gamma_{0}(\alpha\beta))$.

\item[(b)] We show that each $\EuFrak{B}_{\alpha\beta,i}(q)$, where $1\leq i\leq m_{S}$,
is in the space $\S_{4}(\Gamma_{0}(\alpha\beta))$. 
This is obviously the case  
since $\EuFrak{B}_{\alpha\beta,i}(q), 1\leq i\leq m_{S},$ are obtained using an 
exhaustive search which applies items (i)--(v$\prime$) in 
\autoref{ligozat_theorem}. 

Since the dimension of $\S_{4}(\Gamma_{0}(\alpha\beta))$ is finite, it 
suffices to show that the set $\EuScript{B}_{S}$
is linearly independent. Suppose that $x_{i}\in\mathbb{C}$ and 
$\underset{i=1}{\overset{m_{S}}{\sum}}\,x_{i}\,\EuFrak{B}_{\alpha\beta,i}(q)=0$. 
Then 
$\underset{i=1}{\overset{m_{S}}{\sum}}\,x_{i}\,\EuFrak{B}_{\alpha\beta,i}(q)
= \underset{n=1}{\overset{\infty}{\sum}}(\,\underset{i=1}{\overset{m_{S}}{\sum}}\,x_{i}\,\EuFrak{b}_{\alpha\beta,i}(n)\,)q^{n} = 0$
which gives the homogeneous system of $m_{S}$ linear equations in $m_{S}$ unknowns:
\begin{equation}  \label{basis-cusp-eqn-sol}
\underset{i=1}{\overset{m_{S}}{\sum}}\EuFrak{b}_{\alpha\beta,i}(n)\,x_{i}= 0,\qquad 
1\leq n\leq m_{S}.
\end{equation}
Two cases arise:
\begin{description}
\item[The smallest degree of $\EuFrak{B}_{\alpha\beta,i}(q)$ is $i$ for each 
$1\leq i\leq m_{S}$] Then the square matrix which corresponds to 
this homogeneous system of $m_{S}$ linear equations is triangular with $1$'s on the 
diagonal. Hence, the determinant of that matrix is $1$ and 
so the unique solution is $x_{i}=0$ for all $1\leq i\leq m_{S}$. 
\item[The smallest degree of $\EuFrak{B}_{\alpha\beta,i}(q)$ is $i$ for  
$1\leq i< m_{S}$] Let $n'$ be the largest positive integer such that 
$1\leq i\leq n'< m_{S}$. Let 
$\EuScript{B}_{S}'=\{\,\EuFrak{B}_{\alpha\beta,i}(q)\,\mid\,1\leq i\leq n'\,\}$ and 
$\EuScript{B}_{S}''=\{\,\EuFrak{B}_{\alpha\beta,i}(q)\,\mid\,n'< i\leq m_{S}\,\}$. Then 
$\EuScript{B}_{S}=\EuScript{B}_{S}'\cup\EuScript{B}_{S}''$ and we may consider 
$\EuScript{B}_{S}$ as an ordered set. By the case above, the set 
$\EuScript{B}_{S}'$ is linearly 
independent. Hence, the linear independence of the set $\EuScript{B}_{S}$ depends 
on that of the set $\EuScript{B}_{S}''$.
Let $A=(\EuFrak{b}_{\alpha\beta,i}(n))$ be the $m_{S}\times m_{S}$ matrix 
in \hyperref[basis-cusp-eqn-sol]{Equation \ref*{basis-cusp-eqn-sol}}. 
If $\text{det}(A)\neq 0$, then $x_{i}=0$ for all $1\leq i\leq m_{S}$ and we are 
done. Suppose that $\text{det}(A)= 0$. Then for 
some $n'< k\leq m_{S}$ there exists $\EuFrak{B}_{\alpha\beta,k}(q)$ which is causing 
the system of equations to be inconsistent. We substitute $\EuFrak{B}_{\alpha\beta,k}(q)$ with, say 
    $\EuFrak{B}_{\alpha\beta,k}'(q)$, which does not occur in $\EuScript{B}_{S}$ and 
    compute of the determinant of the new matrix $A$. 
    Since there are finitely many 
    $\EuFrak{B}_{\alpha\beta,k}(q)$ with $n'< k\leq m_{S}$ that may cause the 
    system of linear equations to be inconsistent and finitely many elements of 
    $\S_{4}(\Gamma_{0}(\alpha\beta))\setminus \EuScript{B}_{S}$, 
    the procedure will terminate with a consistent system of linear equations.  
Since \aP\ \cite{pizer1976} has proved the existence of a basis for the space 
of cusps, we will find a basis of $\S_{4}(\Gamma_{0}(\alpha\beta))$.
\end{description}
Therefore, the set $\{\,\EuFrak{B}_{\alpha\beta,i}(q)\mid 1\leq i\leq m_{S}\,\}$ 
is linearly independent and hence  
is a basis of $\S_{4}(\Gamma_{0}(\alpha\beta))$.

\item[(c)] Since $\M_{4}(\Gamma_{0}(\alpha\beta))=\E_{4}(\Gamma_{0}(\alpha\beta))\oplus 
\S_{4}(\Gamma_{0}(\alpha\beta))$, the result follows from (a) and (b).
\end{enumerate} 
\end{proof}

Note that if $\EuScript{C}=\emptyset$, that means that the Dirichlet character 
is trivially one, then Theorem 3.1 proved by \eN\ \cite{ntienjem2016c} 
is obtained as an immediate corollary.

\subsection{Evaluating the convolution sum $W_{(\alpha,\beta)}(n)$}
\label{convolution_alpha_beta-gen}

We recall 
the assumption that $\chi\neq 1$ 
since the case $\chi=1$ has been discussed by \eN\ \cite{ntienjem2016c}. 

\begin{lemma} \label{convolution-lemma_a_b}
Let $\alpha,\beta\in\mathbb{N}$ be such that $\text{gcd}(\alpha,\beta)=1$. 
Let furthermore $\EuScript{B}_{M}=\EuScript{B}_{E}\cup\EuScript{B}_{S}$ be a 
basis of $\M(\Gamma_{0}(\alpha\beta))$. Then there exist
$X_{\delta},Z(\chi)_{s}, Y_{j}\in\mathbb{C}, 1\leq j\leq m_{S}, 
\chi\in\EuScript{C}, s\in D(\chi)\text{ and } \delta|\alpha\beta$ such that  
\begin{multline}
(\alpha\, L(q^{\alpha}) - \beta\, L(q^{\beta}))^{2} 
 = \sum_{\delta|\alpha\beta}X_{\delta} + \sum_{\chi\in\EuScript{C}}\sum_{s\in
   D(\chi)}C_{0}Z(\chi)_{s}  \\
   +  \sum_{n=1}^{\infty}\biggl( 
 240\sum_{\delta|\alpha\beta}\sigma_{3}(\frac{n}{\delta})X_{\delta}  
+ \sum_{\chi\in\EuScript{C}}\,\sum_{s\in D(\chi)}\,\sigma_{3}(\frac{n}{s})\,
\chi\,Z(\chi)_{s} 
 + \sum_{j=1}^{m_{S}}\,\EuFrak{b}_{j}(n)\,Y_{j}\biggr)\,q^{n}.
\label{convolution_a_b-eqn-0}
\end{multline}
\end{lemma}  
\begin{proof} That $( \alpha L(q^{\alpha}) - \beta L(q^{\beta}) )^{2}\in
\M_{4}(\Gamma_{0}(\alpha\beta))$ follows from  
\hyperref[evalConvolClass-lema-1]{Lemma \ref*{evalConvolClass-lema-1}}.  
Hence, by 
\autoref{basisCusp_a_b}\,(c), there exist 
$X_{\delta},Z(\chi)_{s}, Y_{j}\in\mathbb{C}, 
1\leq j\leq m_{S}, \chi\in\EuScript{C}, s\in D(\chi)\text{ and }
\delta|\alpha\beta$, such that  
\begin{multline*}
( \alpha L(q^{\alpha}) - \beta L(q^{\beta}) )^{2}   =
\sum_{\delta|\alpha\beta}X_{\delta}\,M(q^{\delta}) + \sum_{\chi\in\EuScript{C}}\,
\sum_{s\in D(\chi)}Z(\chi)_{s}\,M_{\chi}(q^{s}) \\ 
+ \sum_{j=1}^{m_{S}}Y_{j}\,\EuFrak{B}_{j}(q) 
= \sum_{\delta|\alpha\beta}X_{\delta} + \sum_{\chi\in\EuScript{C}}\,
\sum_{s\in D(\chi)}\,C_{0}\,Z(\chi)_{s}
+ \sum_{n=1}^{\infty}\biggl(\, 
 240\,\sum_{\delta|\alpha\beta}\,\sigma_{3}(\frac{n}{\delta})\,X_{\delta}  \\
 + \sum_{\chi\in\EuScript{C}}\,\sum_{s\in D(\chi)}
 \,\chi\,\sigma_{3}(\frac{n}{s})\,Z(\chi)_{s} 
  +  \sum_{j=1}^{m_{S}}\,\EuFrak{b}_{j}(n)\,Y_{j}\, \biggr)\,q^{n}. 
 \end{multline*} 
We equate the right hand side of 
\autoref{convolution_a_b-eqn-0} with that of 
\autoref{evalConvolClass-eqn-11} to obtain
\begin{multline*}
\sum_{n=1}^{\infty}\biggl(\, 
 240\,\sum_{\delta|\alpha\beta}X_{\delta}\,\sigma_{3}(\frac{n}{\delta}) + 
  \sum_{\chi\in\EuScript{C}}\biggl(\,\sum_{s\in D(\chi)}\,\chi\,\sigma_{3}
  (\frac{n}{s})\,Z(\chi)_{s}\,\biggr)  
 + \sum_{j=1}^{m_{S}}Y_{j}\,\EuFrak{b}_{j}(n)\, \biggr)\,q^{n} \\  
  = \sum_{n=1}^{\infty}\biggl(\ 240\,\alpha^{2}\,\sigma_{3}
    (\frac{n}{\alpha}) 
    + 240\,\beta^{2}\,\sigma_{3}(\frac{n}{\beta}) 
    + 48\,\alpha\,(\beta-6n)\,\sigma(\frac{n}{\alpha})  \\
    + 48\,\beta\,(\alpha-6n)\,\sigma(\frac{n}{\beta})   
    - 1152\,\alpha\beta\,W_{(\alpha,\beta)}(n)\,\biggr)q^{n}. 
\end{multline*}
We then take the coefficients of $q^{n}$ such that $n$ is in $D(\alpha\beta)$
and $1\leq n\leq m_{S}$, but as many as the unknown, $X_{1},\ldots,X_{\alpha\beta}$,   
$Z(\chi)_{s}$ for all $\chi\in\EuScript{C}, s\in D(\chi)$,   
and $Y_{1},\ldots,Y_{m_{S}}$, 
to obtain 
a system of $m_{E}+m_{S}$ linear equations whose unique solution determines 
the values of the unknowns. Hence, we obtain the result.
\end{proof}
For the following theorem, let for the sake of simplicity 
$X_{\delta},Z(\chi)_{s}$ and
$Y_{j}$ stand for their values obtained in the previous theorem.
\begin{theorem} \label{convolution_a_b}
Let $n$ be a positive integer. Then 
\begin{align*}
 W_{(\alpha,\beta)}(n)   =  & 
  - \frac{5}{24\,\alpha\beta}\,\sum_{\substack{{\delta|\alpha\beta}\\{\delta\neq\alpha,\,\beta} }}\,
  \sigma_{3}(\frac{n}{\delta})\,X_{\delta}  
 - \frac{1}{1152\,\alpha\beta}\,\sum_{\chi\in\EuScript{C}}\,\sum_{s\in 
 D(\chi)}Z(\chi)_{s}\,\sigma_{3}(\frac{n}{s})  \\ &
+ \frac{5}{24\,\alpha\,\beta}\,(\alpha^{2}-X_{\alpha})\,\sigma_{3}(\frac{n}{\alpha}) 
    + \frac{5}{24\,\alpha\,\beta}\,(\beta^{2}-X_{\beta})\,\sigma_{3}(\frac{n}{\beta})  \\ &
 - \sum_{j=1}^{m_{S}}\,\frac{1}{1152\,\alpha\beta}\,\EuFrak{b}_{j}(n)\,Y_{j}
 + (\frac{1}{24}-\frac{1}{4\beta}n)\sigma(\frac{n}{\alpha}) 
 + (\frac{1}{24}-\frac{1}{4\alpha}n)\sigma(\frac{n}{\beta}).  
\end{align*}
\end{theorem}
\begin{proof}  We equate the right hand side of 
\autoref{convolution_a_b-eqn-0} with that of 
\autoref{evalConvolClass-eqn-11} to yield  
 \begin{align*} 
 1152\,\alpha\beta\,W_{(\alpha,\beta)}(n) = & 
 - 240\,\sum_{\delta|\alpha\beta}\,
 \sigma_{3}(\frac{n}{\delta})\,X_{\delta}  
 - \sum_{\chi\in\EuScript{C}}\,\sum_{s\in D(\chi)}\,Z(\chi)_{s}\,
\sigma_{3}(\frac{n}{s})  \\ &
 - \sum_{j=1}^{m_{S}}\,\EuFrak{b}_{j}(n)\,Y_{j} + 240\,\alpha^{2}\,\sigma_{3}
    (\frac{n}{\alpha}) 
    + 240\,\beta^{2}\,\sigma_{3}(\frac{n}{\beta}) \\ & 
    + 48\,\alpha\,(\beta-6n)\,\sigma(\frac{n}{\alpha})  
    + 48\,\beta\,(\alpha-6n)\,\sigma(\frac{n}{\beta}).   
\end{align*}
We then solve for $W_{(\alpha,\beta)}(n)$ to obtain the stated result.
\end{proof}
 \begin{remark}
\begin{enumerate}
\item[(a)] 
As observed by \eN\ \cite{ntienjem2016c}, the following part of 
\autoref{convolution_a_b} depends only on $n$, $\alpha$ and $\beta$ but 
not on the basis of the modular space $\M_{4}(\Gamma_{0}(\alpha\beta))$: 
$$(\frac{1}{24}-\frac{1}{4\beta}n)\sigma(\frac{n}{\alpha}) 
 + (\frac{1}{24}-\frac{1}{4\alpha}n)\sigma(\frac{n}{\beta}). $$
\item[(b)] If $\EuScript{C}=\emptyset$, that means that the 
Dirichlet character is trivially one, then Theorem 3.2 proved by \eN\  
\cite{ntienjem2016c} is obtained as an immediate corollary of 
\autoref{convolution_a_b}.
\item[(c)] For all $\chi\in\EuScript{C}$ and for all $s\in D(\chi)$ the 
value of $Z(\chi)_{s}$ appears to be zero in all explicit 
examples evaluated as yet. Will the value of $Z(\chi)_{s}$ always  
vanish for all $\alpha\beta$ belonging to this class?
\end{enumerate}
\end{remark}

We now have the prerequisite to determine a formula for the number of representations of a positive integer  $n$  by the Octonary  quadratic form.


\section{Number of Representations of a positive Integer  $n$  by the Octonary  
Quadratic Form  \autoref{introduction-eq-1}
}
\label{representations_a_b}

In this section, we only consider those $\alpha\beta\in\mathbb{N}\setminus\N$ 
such that $\alpha\beta\equiv 0\pmod{4}$. That means, for a given $\kappa\in\mathbb{N}$, 
we restrict the form of $\alpha\beta$ to   
\begin{equation}  \label{representations_a_b-eqn-1}
\alpha\beta=2^{e_{1}}\,\underset{j>1}{\overset{\kappa}
{\prod}}\,p_{j}^{e_{j}}, \quad \text{where $e_{1}\geq 2$ and $e_{j}\geq 2$ for at 
least one $2\leq j\leq\kappa$}. 
\end{equation}

\subsection{Determining $(a,b)\in\mathbb{N}^{2}$}  
\label{deter-representations_a_b}
This approach is similar to the one given by \eN\ \cite{ntienjem2016c}.

Let $\Lambda=2^{e_{1}-2}\,\underset{j>1}{\overset{\kappa}
{\prod}}\,p_{j}^{e_{j}}$,
the set $P=\{\,p_{1}=2^{e_{1}-2}\,\}\cup\{\,p_{j}^{e_{j}}\,|\,1<j\leq\kappa\,\}$, and 
$\EuScript{P}(P)$ be the power set of $P$. Then for each $Q\in\EuScript{P}(P)$ 
we define $\mu(Q)=\underset{q\in Q}{\prod}\,q$. 
We set $\mu(\emptyset)=1$ if $Q=\{\emptyset\}$.
Let now 
\begin{multline*}
\Omega_{4}=\{(\mu(Q_{1}),\mu(Q_{2}))~|~\text{ there exist } 
Q_{1},Q_{2}\in\EuScript{P}(P) \text{ such that }  \\ 
\text{gcd}(\mu(Q_{1}),\mu(Q_{2}))=1\,\text{ and } 
\mu(Q_{1})\,\mu(Q_{2})=\Lambda\,\}.
\end{multline*}
Observe that 
$\Omega_{4}\neq\emptyset$ since $(1,\Lambda)\in\Omega_{4}$.
\begin{proposition} \label{representation-prop-1}
Suppose that $\alpha\beta$ has the above form and suppose that $\Omega_{4}$ is 
defined as above. Then for all $n\in\mathbb{N}$ the set $\Omega_{4}$ 
contains all pairs $(a,b)\in\mathbb{N}^{2}$ such that $N_{(a,b)}(n)$ can 
be obtained by applying $W_{(\alpha,\beta)}(n)$.
\end{proposition}
\begin{proof} Similar to the proof given by \eN\ \cite{ntienjem2016c}.
\end{proof}

\subsection{Number of Representations of a positive Integer}
\label{represent_a_b}
As an immediate application of 
\autoref{convolution_a_b} the number 
of representations of a positive integer $n$  by the octonary quadratic form 
$a\,(x_{1}^{2} +x_{2}^{2} + x_{3}^{2} + x_{4}^{2})
+ b\,(x_{5}^{2} + x_{6}^{2} + x_{7}^{2} + x_{8}^{2})$ 
is determined.

Let $n\in\mathbb{N}$ and the number of representations
of $n$ by the quaternary quadratic form  $x_{1}^{2} +x_{2}^{2}+x_{3}^{2} +
x_{4}^{2}$ be denoted by $r_{4}(n)$. That means, 
\begin{equation*}
r_{4}(n)=\text{card}(\{(x_{1},x_{2},x_{3},x_{4})\in\mathbb{Z}^{4}~|~ m = x_{1}^{2} +x_{2}^{2} + x_{3}^{2} + x_{4}^{2}\}).
\end{equation*}
We set $r_{4}(0) = 1$. \ksW\ \cite{williams2011} has shown that for all
$n\in\mathbb{N}$ 
\begin{equation}
r_{4}(n) = 8\sigma(n) - 32\sigma(\frac{n}{4}). \label{representations-eqn-4-1}
\end{equation}

Now, let the number of representations of $n$ by the octonary quadratic form 
\begin{equation*}
a\,(x_{1}^{2} +x_{2}^{2} + x_{3}^{2} + x_{4}^{2})
+ b\,(x_{5}^{2} + x_{6}^{2} + x_{7}^{2} + x_{8}^{2})
\end{equation*}
be $N_{(a,b)}(n)$. That means,  
\begin{multline*}
N_{(a,b)}(n) 
=\text{card}
(\{(x_{1},x_{2},x_{3},x_{4},x_{5},x_{6},x_{7},x_{8})\in\mathbb{Z}^{8}~|~
n = a\,( x_{1}^{2} +x_{2}^{2}  \\
    + x_{3}^{2} + x_{4}^{2} ) + 
b\,( x_{5}^{2} +x_{6}^{2} + x_{7}^{2} + x_{8}^{2}) \}).
\end{multline*}
We then derive the following result:
\begin{theorem} \label{representations-theor_a_b}
Let $n\in\mathbb{N}$ and $(a,b)\in\Omega_{4}$. Then  
\begin{multline*}
N_{(a,b)}(n)  = 8\,\sigma(\frac{n}{a}) - 32\,\sigma(\frac{n}{4\,a}) 
+ 8\,\sigma(\frac{n}{b}) - 32\,\sigma(\frac{n}{4b}) + 64\,W_{(a,b)}(n) \\
+ 1024\,W_{(a,b)}(\frac{n}{4}) - 256\,\biggl( W_{(4a,b)}(n) + W_{(a,4b)}(n) \biggr).
\end{multline*}
\end{theorem}

\begin{proof} 
It holds that 
\begin{multline*}
N_{(a,b)}(n)  = \sum_{\substack{
{(l,m)\in\mathbb{N}^{2}} \\ {a\,l+b\,m=n}
 }}r_{4}(l)r_{4}(m) 
   = r_{4}(\frac{n}{a})r_{4}(0) + r_{4}(0)r_{4}(\frac{n}{b})  
   + \sum_{\substack{
{(l,m)\in\mathbb{N}^{2}} \\ {a\,l+b\,m=n}
 }}r_{4}(l)r_{4}(m)
\end{multline*}
We use 
\autoref{representations-eqn-4-1} to derive 
\begin{multline*}
N_{(a,b)}(n)  = 8\sigma(\frac{n}{a}) - 32\sigma(\frac{n}{4a}) + 8\sigma(\frac{n}{b}) -
32\sigma(\frac{n}{4b}) \\
   + \sum_{\substack{
{(l,m)\in\mathbb{N}^{2}} \\ {al+bm=n}
  }} (8\sigma(l) - 32\sigma(\frac{l}{4}))(8\sigma(m) - 32\sigma(\frac{m}{4})). 
\end{multline*}
We observe that 
\begin{multline*}
(8\sigma(l) - 32\sigma(\frac{l}{4}))(8\sigma(m) - 32\sigma(\frac{m}{4}))  = 
64\sigma(l)\sigma(m) - 256\sigma(\frac{l}{4})\sigma(m) \\
   - 256\sigma(l)\sigma(\frac{m}{4})  + 1024\sigma(\frac{l}
   {4})\sigma(\frac{m}{4}).
\end{multline*}
We assume in the sequel of this proof that the evaluation of 
\begin{equation*}
W_{(a,b)}(n) = \sum_{\substack{
{(l,m)\in\mathbb{N}^{2}} \\ {al+bm=n}
 }}\sigma(l)\sigma(m), 
\end{equation*}
$W_{(4a,b)}(n)$ and $W_{(a,4b)}(n)$ are known. We map $l$ to $4l$ and 
$m$ to $4m$ to derive  
\begin{equation*}
W_{(4a,b)}(n) = \sum_{\substack{
{(l,m)\in\mathbb{N}^{2}} \\ {al+bm=n}
}}\sigma(\frac{l}{4})\sigma(m) 
 = \sum_{\substack{
{(l,m)\in\mathbb{N}^{2}} \\ {4a\,l+bm=n}
 }}\sigma(l)\sigma(m)
\end{equation*}
and 
\begin{equation*}
W_{(a,4b)}(n) = \sum_{\substack{
{(l,m)\in\mathbb{N}^{2}} \\ {al+bm=n}
}}\sigma(l)\sigma(\frac{m}{4})  = \sum_{\substack{
{(l,m)\in\mathbb{N}^{2}} \\ {al+4b\,m=n}
 }}\sigma(l)\sigma(m),
\end{equation*}
respectively. 
We simultaneously map $l$ to $4l$ and $m$ to $4m$ to infer 
\begin{equation*}
\sum_{\substack{
{(l,m)\in\mathbb{N}^{2}} \\ {al+bm=n}
}}\sigma(\frac{l}{4})\sigma(\frac{m}{4}) 
= \sum_{\substack{
{(l,m)\in\mathbb{N}^{2}} \\ {al+bm=\frac{n}{4}}
 }}\sigma(l)\sigma(m)
 = W_{(a,b)}(\frac{n}{4}).
\end{equation*}

We put these eval\-u\-a\-tions to\-gether to ob\-tain the stated re\-sult for 
$N_{(a,b)}(n)$. 
\end{proof}


\section{Number of Representations of a Positive Integer $n$  by the Octonary 
Quadratic Form \autoref{introduction-eq-2}
} 
\label{representations_c_d}
We now consider those $\alpha\beta\in\mathbb{N}\setminus\N$ 
for which $\alpha\beta\equiv 0\pmod{3}$. That means, for a given $\kappa\in\mathbb{N}$,
we consider the restricted form of $\alpha\beta$ 
\begin{equation}  \label{representations_c_d-eqn-1}
\alpha\beta=2^{e_{1}}\cdot 3^{e_{2}}\,
\underset{j>2}{\overset{\kappa}{\prod}}\,p_{j}^{e_{j}}, \,
\text{where } p_{j}\geq 5 \text{ and at least for one $2\leq j\leq\kappa$ we have } e_{j}\geq 2. 
\end{equation}

\subsection{Determining $(c,d)\in\mathbb{N}^{2}$} 
\label{deter-representations_c_d}
This method is similar to the one given by \eN\ \cite{ntienjem2016c}.
The construction is almost the same as the one given in 
\hyperref[deter-representations_a_b]{Subsection \ref*{deter-representations_a_b}}. 

Let $\Delta=2^{e_{1}}\cdot 3^{e_{2}-1}\,\underset{j>2}{\overset{\kappa}
{\prod}}\,p_{j}^{e_{j}}$,
the set $P=\{p_{1}=2^{e_{1}}, p_{2}=3^{e_{2}-1}\,\}\cup\{\,p_{j}^{e_{j}}\,|\,2\leq j\leq\kappa\,\}$, and 
$\EuScript{P}(P)$ be the power set of $P$. Then for each $Q\in\EuScript{P}(P)$ 
we define $\mu(Q)=\underset{q\in Q}{\prod}\,q$. We set $\mu(\emptyset)=1$ if 
$Q=\{\emptyset\}$.
Let now $\Omega_{3}$ be defined as in 
\hyperref[deter-representations_a_b]{Subsection \ref*{deter-representations_a_b}} 
with $\Delta$ instead of $\Lambda$, i.e., 
\begin{multline*}
\Omega_{3}=\{(\mu(Q_{1}),\mu(Q_{2}))~|~\text{ there exist } Q_{1},Q_{2}\in\EuScript{P}(P) 
\text{ such that }  \\ 
\text{gcd}(\mu(Q_{1}),\mu(Q_{2}))=1\,\text{ and } 
\mu(Q_{1})\,\mu(Q_{2})=\Delta\,\}.
\end{multline*}
Note that 
$\Omega_{3}\neq\emptyset$ since $(1,\Delta)\in\Omega_{3}$.
\begin{proposition} \label{representation-prop-2}
Suppose that $\alpha\beta$ has the above form and Suppose that $\Omega_{3}$ 
be defined as above. Then for all $n\in\mathbb{N}$ the set $\Omega_{3}$ 
contains all pairs $(c,d)\in\mathbb{N}^{2}$ such that $R_{(c,d)}(n)$ can 
be obtained by applying $W_{(\alpha,\beta)}(n)$.
\end{proposition}
\begin{proof} 
Simlar to the proof of 
\hyperref[representation-prop-1]{Proposition \ref*{representation-prop-1}}.
\end{proof}

\subsection{Number of Representations of a positive Integer}
\label{represent_c_d}
The number 
of representations of a positive integer $n$  by the octonary quadratic form  
$c\,(x_{1}^{2} + x_{1}x_{2} + x_{2}^{2} + x_{3}^{2} + x_{3}x_{4} + x_{4}^{2})
+ d\,(x_{5}^{2} + x_{5}x_{6}+ x_{6}^{2} + x_{7}^{2} + x_{7}x_{8}+ x_{8}^{2})$
is determined as an immediate application of 
\autoref{convolution_a_b}. 

Let $n\in\mathbb{N}$ and  let $s_{4}(n)$ denote the number of representations
of $n$ by the quaternary quadratic form  $x_{1}^{2} + x_{1}x_{2} + x_{2}^{2} +
x_{3}^{2} + x_{3}x_{4} + x_{4}^{2}$, that is  
\begin{equation*}
s_{4}(n)=\text{card}(\{(x_{1},x_{2},x_{3},x_{4})\in\mathbb{Z}^{4}~|~ n = x_{1}^{2} + x_{1}x_{2} + x_{2}^{2} + x_{3}^{2} + x_{3}x_{4} + x_{4}^{2}\}).
\end{equation*}
We set $s_{4}(0) = 1$. \jgH\ et al. \cite{huardetal} and \gaL\ 
\cite{lomadze} have proved that for all
$n\in\mathbb{N}$ 
\begin{equation}
s_{4}(n) = 12\sigma(n) - 36\sigma(\frac{n}{3}). \label{representations-eqn-c_d-1}
\end{equation}
Now, let the number of representations of $n$ by the octonary quadratic form 
\begin{equation*}
c\,(x_{1}^{2} + x_{1}x_{2} + x_{2}^{2} + x_{3}^{2} + x_{3}x_{4} + x_{4}^{2})
+ d\,(x_{5}^{2} + x_{5}x_{6}+ x_{6}^{2} + x_{7}^{2} + x_{7}x_{8}+ 
x_{8}^{2})
\end{equation*}
by $R_{(c,d)}(n)$. That is, 
\begin{multline*}
R_{(c,d)}(n) =\text{card}
(\{(x_{1},x_{2},x_{3},x_{4},x_{5},x_{6},x_{7},x_{8})\in\mathbb{Z}^{8}~|~
n = c\,(x_{1}^{2} + x_{1}x_{2} + x_{2}^{2}  \\
    + x_{3}^{2} + x_{3}x_{4} + x_{4}^{2}) + 
d\,( x_{5}^{2} + x_{5}x_{6}+ x_{6}^{2} + x_{7}^{2} + x_{7}x_{8}+ x_{8}^{2}) \}).
\end{multline*}
We infer the following result.
\begin{theorem} \label{representations-theor-c_d}
Let $n\in\mathbb{N}$ and $(c,d)\in\Omega_{3}$. Then  
\begin{multline*}
R_{(c,d)}(n) = 12\,\sigma(\frac{n}{c}) - 36\,\sigma(\frac{n}{3c}) 
+ 12\,\sigma(\frac{n}{d}) - 36\,\sigma(\frac{n}{3d}) + 144\,W_{(c,d)}(n) \\
   + 1296\,W_{(c,d)}(\frac{n}{3}) - 432\,\biggl( W_{(3c,d)}(n) + 
   W_{(c,3d)}(n) \biggr). 
\end{multline*}
\end{theorem}
\begin{proof} It is obvious that 
\begin{multline*}
R_{(c,d)}(n) = \sum_{\substack{
{(l,m)\in\mathbb{N}^{2}} \\ {cl+dm=n}
}}s_{4}(l)s_{4}(m) 
  = s_{4}(\frac{n}{c})s_{4}(0) + s_{4}(0)s_{4}(\frac{n}{d})  
+ \sum_{\substack{
{(l,m)\in\mathbb{N}^{2}} \\ {cl+dm=n}
}}s_{4}(l)s_{4}(m). 
\end{multline*}
We make use of 
\autoref{representations-eqn-c_d-1} to deduce 
\begin{multline*}
R_{(c,d)}(n) = 12\sigma(\frac{n}{c}) - 36\sigma(\frac{n}{3c}) + 12\sigma(\frac{n}{d}) -
36\sigma(\frac{n}{3d}) \\
   + \sum_{\substack{
{(l,m)\in\mathbb{N}^{2}} \\ {cl+dm=n}
 }} (12\sigma(l) -
  36\sigma(\frac{l}{3}))(12\sigma(m) - 36\sigma(\frac{m}{3})).
\end{multline*}
We observe that 
\begin{multline*}
(12\sigma(l) - 36\sigma(\frac{l}{3}))(12\sigma(m) - 36\sigma(\frac{m}{3})) = 
144\sigma(l)\sigma(m) - 432\sigma(\frac{l}{3})\sigma(m) \\
   - 432\sigma(l)\sigma(\frac{m}{3})  + 1296\sigma(\frac{l}{3})\sigma(\frac{m}{3}).
\end{multline*}
We assume that the evaluation of 
\begin{equation*}
 W_{(c,d)}(n) = \sum_{\substack{
{(l,m)\in\mathbb{N}^{2}} \\ {cl+dm=n}
 }}\sigma(l)\sigma(m), 
\end{equation*}
$W_{(3c,d)}(n)$ and $W_{(c,3d)}(n)$ are known.  
We apply the transformations $m$ to $3m$ and $l$ to $3l$ to infer 
\begin{equation*}
\sum_{\substack{
{(l,m)\in\mathbb{N}^{2}} \\ {cl+dm=n}
 }}\sigma(l)\sigma(\frac{m}{3})  = 
 \sum_{\substack{
{(l,m)\in\mathbb{N}^{2}} \\ {cl+3d\,m=n}
 }}\sigma(l)\sigma(m)
 = W_{(c,3d)}(n)  
\end{equation*}
and 
\begin{equation*}
\sum_{\substack{
{(l,m)\in\mathbb{N}^{2}} \\ {cl+dm=n}
 }}\sigma(m)\sigma(\frac{l}{3})  = 
\sum_{\substack{
{(l,m)\in\mathbb{N}^{2}} \\ {3c\,l+dm=n}
 }}\sigma(l)\sigma(m)
= W_{(3c,d)}(n),
\end{equation*}
respectively. 
We simultaneously map $l$ to $3l$ and $m$ to $3m$ deduce
\begin{equation*}
\sum_{\substack{
{(l,m)\in\mathbb{N}^{2}} \\ {cl+dm=n}
 }}\sigma(\frac{m}{3})\sigma(\frac{l}{3})  = 
\sum_{\substack{
{(l,m)\in\mathbb{N}^{2}} \\ {cl+dm=\frac{n}{3}}
 }}\sigma(l)\sigma(m)
  = W_{(c,d)}(\frac{n}{3}).
\end{equation*}

We put these evaluations together to obtain the stated result 
for $R_{(c,d)}(n)$.
\end{proof}


\section{Evaluation of the convolution sums $W_{(\alpha,\beta)}(n)$ when $\alpha\beta=45,48, 50, 64$}
\label{convolution_45_50}

In this Section, we give explicit formulae for the convolution sums 
$W_{(1,45)}(n)$, $W_{(5,9)}(n)$, $W_{(1,48)}(n)$, $W_{(3,16)}(n)$, $W_{(1,50)}(n)$,  
$W_{(2,25)}(n)$ and $W_{(1,64}(n)$. 

The two convolution sums $W_{(1,50)}(n)$ and $W_{(2,25)}(n)$  are worth mentioning 
due to the fact that the set of divisors of $50$ which are associated with the 
Dirichlet character for the formation of a basis of the space of Eisensten forms is 
the whole set of divisors of $50$.

\subsection{Bases for $\E_{4}(\Gamma_{0}(\alpha\beta))$ and $\S_{4}(\Gamma_{0}(\alpha\beta))$ when $\alpha\beta=45,48,50, 64$}  
\label{convolution_45_50-gen}

The dimension formulae for the space of 
cusp forms as given in T.~Miyake's book  
\cite[Thrm 2.5.2,~p.~60]{miyake1989} or W.~A.~Stein's book  
\cite[Prop.\ 6.1, p.\ 91]{wstein} and \autoref{dimension-Eisenstein} are applied to compute 
\begin{align*}   
\text{dim}(\E_{4}(\Gamma_{0}(45)))=8, \quad 
\text{dim}(\S_{4}(\Gamma_{0}(45)))=14,  \\
\text{dim}(\E_{4}(\Gamma_{0}(48)))=\text{dim}(\E_{4}(\Gamma_{0}(50)))=\text{dim}(\E_{4}(\Gamma_{0}(64)))=12  \\ 
\text{dim}(\S_{4}(\Gamma_{0}(50)))=17, \quad 
\text{dim}(\S_{4}(\Gamma_{0}(48))=\text{dim}(\S_{4}(\Gamma_{0}(64))=18. 
\end{align*}

We use \autoref{ligozat_theorem} to determine $\eta$-quotients which are 
elements of $\S_{4}(\Gamma_{0}(45))$, $\S_{4}(\Gamma_{0}(48))$,  
$\S_{4}(\Gamma_{0}(50))$ and $\S_{4}(\Gamma_{0}(64))$, respectively.
 
Let $D(45)$, $D(48)$, $D(50)$ and $D(64)$ denote the sets 
of positive divisors of $45$, $48$, $50$ and $64$, respectively.

We observe that 
\begin{align}
\M_{4}(\Gamma_{0}(5)) \subset \M_{4}(\Gamma_{0}(15)) \subset \M_{4}(\Gamma_{0}(45)) 
 \label{bas-s-eqn-5_45} \\
\M_{4}(\Gamma_{0}(9)) \subset \M_{4}(\Gamma_{0}(45)) \label{bas-s-eqn-9_45} \\
\M_{4}(\Gamma_{0}(6)) \subset \M_{4}(\Gamma_{0}(12)) \subset \M_{4}(\Gamma_{0}(24)) 
\subset \M_{4}(\Gamma_{0}(48)) \label{bas-s-eqn-6_48} \\
\M_{4}(\Gamma_{0}(8)) \subset \M_{4}(\Gamma_{0}(24)) \subset \M_{4}(\Gamma_{0}(48)) \label{bas-s-eqn-8_48-1} \\
\M_{4}(\Gamma_{0}(8)) \subset \M_{4}(\Gamma_{0}(16)) \subset \M_{4}(\Gamma_{0}(48)) \label{bas-s-eqn-8_48-2} \\
\M_{4}(\Gamma_{0}(5)) \subset \M_{4}(\Gamma_{0}(25)) \subset \M_{4}(\Gamma_{0}(50)) \label{bas-s-eqn-5_50-1}\\
\M_{4}(\Gamma_{0}(5)) \subset \M_{4}(\Gamma_{0}(10)) \subset \M_{4}(\Gamma_{0}(50)) \label{bas-s-eqn-5_50-2}  \\
\M_{4}(\Gamma_{0}(8)) \subset \M_{4}(\Gamma_{0}(16)) \subset \M_{4}(\Gamma_{0}(32)) 
\subset \M_{4}(\Gamma_{0}(64)). \label{bas-s-eqn-8_64}
\end{align}
A graphical illustration of the inclusion relation represented by 
\autoref{bas-s-eqn-5_45}, \autoref{bas-s-eqn-9_45}, \autoref{bas-s-eqn-5_50-1} 
and \autoref{bas-s-eqn-5_50-2} are  given 
in \hyperref[bas-s-45-50]{Figure \ref*{bas-s-45-50}}; and that represented by 
\autoref{bas-s-eqn-6_48}, \autoref{bas-s-eqn-8_48-1}, and 
\autoref{bas-s-eqn-8_48-2} in \hyperref[bas-s-48]{Figure \ref*{bas-s-48}}.

\begin{corollary} \label{basisCusp5_9}
\begin{enumerate}
\item[\textbf{(a)}] The sets 
\begin{align*}
\EuScript{B}_{E,45}=\{M(q^{t})\mid t|45\}\cup\{\,M_{\legendre{-4}{n}}(q^{s})\mid 
s=1,3\}   \\ 
\EuScript{B}_{E,48}=\{\,M(q^{t})\,\mid ~t\in D(48)\,\}
\cup\{\,M_{\legendre{-3}{n}}(q^{s})\mid s=1,2\}  \\
\EuScript{B}_{E,50}=\{M(q^{t})\mid t|50\}\cup\{\,M_{\legendre{-3}{n}}(q^{s})\mid 
s\in D(50)\,\}   \quad\text{and}\quad \\ 
\EuScript{B}_{E,64}=\{\, M(q^{t})\,\mid ~ t\in D(64)\,\} 
\cup\{\,M_{\legendre{-3}{n}}(q^{s})\mid s=1,2,4,8,16\}
\end{align*} 
constitute bases of $\E_{4}(\Gamma_{0}(45))$, $\E_{4}(\Gamma_{0}(48))$, 
$\E_{4}(\Gamma_{0}(50))$ and $\E_{4}(\Gamma_{0}(64))$, respectively.
\item[\textbf{(b)}] Let $1\leq i\leq 14$, $1\leq j\leq 17$, $1\leq k,l\leq 18$ 
be positive integers.

Let $\delta_{1}\in D(45)$ and 
$(r(i,\delta_{1}))_{i,\delta_{1}}$ be the  
\autoref{convolutionSums-5_9-table} of the powers of $\eta(\delta_{1}\,z)$. 

Let $\delta_{3}\in D(48)$ and 
$(r(k,\delta_{3}))_{k,\delta_{3}}$ be the 
\autoref{convolutionSums-3_16-table} of the powers of $\eta(\delta_{3} z)$. 

Let $\delta_{2}\in D(50)$ and 
$(r(j,\delta_{2}))_{j,\delta_{2}}$ be the  
\autoref{convolutionSums-2_25-table} of the powers of $\eta(\delta_{2}\,z)$. 

Let $\delta_{4}\in D(64)$ and 
$(r(l,\delta_{4}))_{l,\delta_{4}}$ be the  
\autoref{convolutionSums-1_64-table} of the powers of $\eta(\delta_{4} z)$. 

Let furthermore  
\begin{align*}
\EuFrak{B}_{45,i}(q)=\underset{\delta_{1}|45}{\prod}\eta^{r(i,\delta_{1})}(\delta_{1}
\,z),   \qquad 
\EuFrak{B}_{48,k}(q)=\underset{\delta_{3}|48}{\prod}\eta^{r(k,\delta_{3})}(\delta_{3}
\,z),  \\
\EuFrak{B}_{50,j}(q)=\underset{\delta_{2}|50}{\prod}\eta^{r(j,\delta_{2})}(\delta_{2}
\,z),  \qquad 
\EuFrak{B}_{64,l}(q)=\underset{\delta_{4}|64}{\prod}\eta^{r(l,\delta_{4})}(\delta_{4}
\,z)  
\end{align*} 
be selected elements of 
$\S_{4}(\Gamma_{0}(45))$, $\S_{4}(\Gamma_{0}(48))$, $\S_{4}(\Gamma_{0}(50))$ and 
$\S_{4}(\Gamma_{0}(64))$, respectively. 

The sets 
\begin{align*}
\EuScript{B}_{S,45}=\{ \EuFrak{B}_{45,i}(q)\mid ~1\leq i\leq 14\}, \qquad   
\EuScript{B}_{S,48}=\{ \EuFrak{B}_{48,k}(q)\mid ~1\leq k\leq 18\},  \\
\EuScript{B}_{S,50}=\{ \EuFrak{B}_{50,j}(q)\mid ~1\leq j\leq 17\}, \qquad
\EuScript{B}_{S,64}=\{ \EuFrak{B}_{64,l}(q)\mid ~1\leq l\leq 18\}
\end{align*} 
are bases of $\S_{4}(\Gamma_{0}(45))$, $\S_{4}(\Gamma_{0}(48))$,  
$\S_{4}(\Gamma_{0}(50))$ and $\S_{4}(\Gamma_{0}(64))$ , respectively.
\item[\textbf{(c)}] The sets 
\begin{align*}
\EuScript{B}_{M,45}=\EuScript{B}_{E,45}\cup\EuScript{B}_{S,45}, \quad 
\EuScript{B}_{M,48}=\EuScript{B}_{E,48}\cup\EuScript{B}_{S,48},  \\ 
\EuScript{B}_{M,50}=\EuScript{B}_{E,50}\cup\EuScript{B}_{S,50},  \quad 
\EuScript{B}_{M,64}=\EuScript{B}_{E,64}\cup\EuScript{B}_{S,64}
\end{align*} 
constitute bases of $\M_{4}(\Gamma_{0}(45))$, $\M_{4}(\Gamma_{0}(48))$, 
$\M_{4}(\Gamma_{0}(50))$ and $\M_{4}(\Gamma_{0}(64))$,  respectively. 
\end{enumerate}
\end{corollary}

By \hyperref[basis-remark]{Remark \ref*{basis-remark}} (r1), each 
$\EuFrak{B}_{\alpha\beta,i}(q)$ is expressible in the form 
$\underset{n=1}{\overset{\infty}{\sum}}\EuFrak{b}_{\alpha\beta,i}(n)q^{n}$.

\begin{proof} We only give the proof for $\EuScript{B}_{M,45}=\EuScript{B}_{E,45}\cup
\EuScript{B}_{S,45}$ since the other cases are done similarly. In the case of 
$\EuScript{B}_{E,48}$, $\EuScript{B}_{E,50}$, $\EuScript{B}_{E,64}$ 
an applicable primitive Dirichlet character is 
\begin{equation} \label{base-3_16-kronecker}
\legendre{-3}{n}= \begin{cases}
 -1 & \text{ if } n\equiv 2\pmod{3}, \\
 0 & \text{ if } \text{gcd}(3,n) \neq 1, \\
 1 & \text{ if } n\equiv 1\pmod{3}. 
	\end{cases}
\end{equation}
 
\begin{enumerate}
\item[\textbf{(a)}]
Suppose that $x_{\delta},z_{1},z_{3}\in\mathbb{C}$ with $\delta|45$. Let   
$$\underset{\delta|45}{\sum} x_{\delta}\,M(q^{\delta}) + 
z_{1}\,M_{\legendre{-4}{n}}(q) + z_{3}\,M_{\legendre{-4}{n}}(q^{3})=0.$$ 
We observe that 
\begin{equation} \label{base-5_9-kronecker}
\legendre{-4}{n}= \begin{cases}
 -1 & \text{ if } n\equiv 3\pmod{4}, \\
 0 & \text{ if } \text{gcd}(4,n) \neq 1, \\
 1 & \text{ if } n\equiv 1\pmod{4}. 
	\end{cases}
\end{equation}
and recall that for all $0\neq a\in\mathbb{Z}$ it holds that $\legendre{a}{0}=0$. 
Since the conductor of the Dirichlet character $\legendre{-4}{n}$ is $4$, we infer 
from \autoref{Eisenstein-gen} that $C_{0}=0$. 
We then deduce  
\begin{equation*}
\underset{\delta|45}{\sum} x_{\delta} 
 + \underset{i=1}{\overset{\infty}
{\sum}}\biggl(240\underset{\delta|45}{\sum}\sigma_{3}(\frac{n}{\delta}) x_{\delta} + 
\legendre{-4}{n}\sigma_{3}(n)z_{1} + \legendre{-4}{n}\sigma_{3}(\frac{n}
{3})z_{3}\biggr)q^{n}=0.  
\end{equation*}
Then we equate the coefficients of $q^{n}$ for $n\in D(45)$ plus for example  
$n=2,7$ 
to obtain a system of $8$ linear equations 
whose unique solution is $x_{\delta}=z_{1}=z_{3}=0$ with $\delta\in D(45)$. So, the set 
$\EuScript{B}_{E}$ is linearly independent. 
Hence, the set $\EuScript{B}_{E}$ is a basis of 
$\E_{4}(\Gamma_{0}(45))$.
\item[\textbf{(b)}] 
Suppose that $x_{i}\in\mathbb{C}$ with $1\leq i\leq 14$. Let  
$\underset{i=1}{\overset{14}{\sum}}x_{i}\,\EuFrak{B}_{45,i}(q)=0$. Then   
\begin{equation*}
\underset{i=1}{\overset{14}{\sum}}x_{i}\underset{n=1}{\overset{\infty}{\sum}}\,\EuFrak{b}_{45,i}(n)q^{n}
= \underset{n=1}{\overset{\infty}{\sum}}\biggl(\,\underset{i=1}{\overset{14}{\sum}}\,\EuFrak{b}_{45,i}(n)\,x_{i}\,\biggr)q^{n} = 0.
\end{equation*}
So, we equate the coefficients of $q^{n}$ for $1\leq n\leq 14$ to obtain 
a system of $14$ linear equations 
whose unique solution is $x_{i}=0$ for all  $1\leq i\leq 14$. 
It follows that the set $\EuScript{B}_{S}$ is linearly independent. 
Hence, the set $\EuScript{B}_{S}$ is a basis of 
$S_{4}(\Gamma_{0}(45))$.
\item[\textbf{(c)}]
Since $\M_{4}(\Gamma_{0}(45))=\E_{4}(\Gamma_{0}(45))\oplus 
\S_{4}(\Gamma_{0}(45))$, the result follows from (a) and (b).
\end{enumerate}
\end{proof}

\subsection{Evaluation of $W_{(\alpha,\beta)}(n)$ for  $\alpha\beta=45,48,50,64$} 
\label{convolSum-w_45_50}  

We evaluate the convolution sums $W_{(\alpha,\beta)}(n)$ for  
$(\alpha,\beta)=(1,45)$, $(5,9)$, $(1,48)$,$(3,16)$,$(1,50)$, $(1,64)$. 

\begin{corollary} \label{lema_45_50}
It holds that 
\begin{multline}
(5\, L(q^{5}) - 9\, L(q^{9}))^{2} 
 = 16 + \sum_{n=1}^{\infty}\biggl(\, - \frac{120}{13}\sigma_{3}(n) 
	- \frac{51960}{923}\sigma_{3}(\frac{n}{3}) \\
   + \frac{75000}{13}\sigma_{3}(\frac{n}{5})   
    + \frac{1296000}{71}\sigma_{3}(\frac{n}{9}) 
   - \frac{5089800}{923}\sigma_{3}(\frac{n}{15})  
    + \frac{5184000}{71}\sigma_{3}(\frac{n}{45}) \\
   - \frac{19344}{1349}\,\EuFrak{b}_{45,1}(n)    
   + \frac{239256}{1349}\,\EuFrak{b}_{45,2}(n) 
   + \frac{10760952}{17537}\,\EuFrak{b}_{45,3}(n) 
   + \frac{762672}{1349}\,\EuFrak{b}_{45,4}(n)  \\
   - \frac{2459904}{1349}\,\EuFrak{b}_{45,5}(n) 
   + \frac{1247280}{1349}\,\EuFrak{b}_{45,6}(n) 
   + \frac{5755968}{1349}\,\EuFrak{b}_{45,7}(n) 
   + \frac{370080}{71}\,\EuFrak{b}_{45,8}(n)  \\
   + \frac{2503632}{1349}\,\EuFrak{b}_{45,9}(n) 
   + \frac{302400}{71}\,\EuFrak{b}_{45,10}(n) 
   - \frac{14389920}{1349}\,\EuFrak{b}_{45,11}(n) \\
   + \frac{413352}{17537}\,\EuFrak{b}_{45,12}(n) 
   + \frac{11760}{1349}\,\EuFrak{b}_{45,13}(n) 
\,\biggr)q^{n}.
\label{convolSum-eqn-5_9}
\end{multline}
\begin{multline}
(L(q) - 45\, L(q^{45}))^{2} 
 = 1936 + \sum_{n=1}^{\infty}\biggl(\, - \frac{120}{13}\sigma_{3}(n) 
	- \frac{51960}{923}\sigma_{3}(\frac{n}{3}) 
   + \frac{75000}{13}\sigma_{3}(\frac{n}{5})  \\ 
    + \frac{1296000}{71}\sigma_{3}(\frac{n}{9}) 
   - \frac{5089800}{923}\sigma_{3}(\frac{n}{15})  
    + \frac{5184000}{71}\sigma_{3}(\frac{n}{45}) 
   - \frac{19344}{1349} \,\EuFrak{b}_{45,1}(n) \\
   + \frac{239256}{1349} \,\EuFrak{b}_{45,2}(n) 
   + \frac{10760952}{17537}\,\EuFrak{b}_{45,3}(n) 
   + \frac{762672}{1349} \,\EuFrak{b}_{45,4}(n)
   - \frac{2459904}{1349} \,\EuFrak{b}_{45,5}(n) \\
   + \frac{1247280}{1349} \,\EuFrak{b}_{45,6}(n) 
   + \frac{5755968}{1349}\,\EuFrak{b}_{45,7}(n) 
   + \frac{370080}{71} \,\EuFrak{b}_{45,8}(n)  
   + \frac{2503632}{1349} \,\EuFrak{b}_{45,9}(n) \\
   + \frac{302400}{71} \,\EuFrak{b}_{45,10}(n) 
   - \frac{14389920}{1349}\,\EuFrak{b}_{45,11}(n) 
   + \frac{413352}{17537} \,\EuFrak{b}_{45,12}(n)
   + \frac{11760}{1349}\,\EuFrak{b}_{45,13}(n) 
\,\biggr)q^{n}.
\label{convolSum-eqn-1_45}
\end{multline}
\begin{multline}
( L(q) - 48\, L(q^{48}))^{2}  
 = 2209 + \sum_{n=1}^{\infty}\biggl(\, 
 \frac{1164}{5}\,\sigma_{3}(n) 
 -\frac{20412}{65}\,\sigma_{3}(\frac{n}{2})
 -\frac{324}{5}\,\sigma_{3}(\frac{n}{3})  \\
 -\frac{290736}{65}\,\sigma_{3}(\frac{n}{4})
+  \frac{6372}{65}\,\sigma_{3}(\frac{n}{6})
+ \frac{281664}{65}\,\sigma_{3}(\frac{n}{8})
+ \frac{234576}{65}\,\sigma_{3}(\frac{n}{12})
 -\frac{9216}{5}\,\sigma_{3}(\frac{n}{16}) \\
 -\frac{506304}{65}\,\sigma_{3}(\frac{n}{24})
+\frac{2681856}{5}\,\sigma_{3}(\frac{n}{48})
+  \frac{38772}{65}\,\EuFrak{b}_{48,1}(n)
+ \frac{639792}{65}\,\EuFrak{b}_{48,2}(n)  \\
+ \frac{546804}{65}\,\EuFrak{b}_{48,3}(n)  
+ \frac{661824}{65}\,\EuFrak{b}_{48,4}(n)
+ \frac{195264}{65}\,\EuFrak{b}_{48,5}(n)
+ \frac{3729456}{65}\,\EuFrak{b}_{48,6}(n)  \\
+ \frac{67968}{5}\,\EuFrak{b}_{48,7}(n)
+ \frac{5422464}{65}\,\EuFrak{b}_{48,8}(n)  
+ \frac{4414464}{65}\,\EuFrak{b}_{48,9}(n)
+ \frac{3151872}{65}\,\EuFrak{b}_{48,10}(n)  \\
 +\frac{1426176}{65}\,\EuFrak{b}_{48,11}(n)
+\frac{14145408}{65}\,\EuFrak{b}_{48,12}(n)  
-\frac{3907584}{65}\,\EuFrak{b}_{48,13}(n)  
-\frac{1693440}{13}\,\EuFrak{b}_{48,14}(n)  \\
+ \frac{313344}{65}\,\EuFrak{b}_{48,15}(n)
+ \frac{7299072}{13}\,\EuFrak{b}_{48,16}(n) 
+ \frac{1032192}{13}\,\EuFrak{b}_{48,17}(n)  
+  \frac{92736}{65}\,\EuFrak{b}_{48,18}(n) 
  \,\biggr)q^{n}, \label{convolSum-eqn-1_48}  
\end{multline}
\begin{multline}
( 3\,L(q^{3}) - 16\, L(q^{16}))^{2}  
 = 169 + \sum_{n=1}^{\infty}\biggl(\, 
 -\frac{36}{5}\,\sigma_{3}(n)   
+  \frac{43092}{1885}\,\sigma_{3}(\frac{n}{2}) \\
+ \frac{10476}{5}\,\sigma_{3}(\frac{n}{3})  
+ \frac{1094256}{1885}\,\sigma_{3}(\frac{n}{4})  
 -\frac{450252}{1885}\,\sigma_{3}(\frac{n}{6})
 -\frac{1992384}{1885}\,\sigma_{3}(\frac{n}{8}) \\
 -\frac{2722896}{1885}\,\sigma_{3}(\frac{n}{12}) 
+ \frac{297984}{5}\,\sigma_{3}(\frac{n}{16})
 -\frac{4522176}{1885}\,\sigma_{3}(\frac{n}{24})
 -\frac{82944}{5}\,\sigma_{3}(\frac{n}{48}) \\
  -\frac{34668}{65}\,\EuFrak{b}_{48,1}(n)  
+ \frac{3135888}{1885}\,\EuFrak{b}_{48,2}(n)
 -\frac{140076}{65}\,\EuFrak{b}_{48,3}(n)
 -\frac{4567104}{1885}\,\EuFrak{b}_{48,4}(n)  \\
 -\frac{115776}{65}\,\EuFrak{b}_{48,5}(n)   
+  \frac{20304}{1885}\,\EuFrak{b}_{48,6}(n)  
+ \frac{91008}{5}\,\EuFrak{b}_{48,7}(n)  
-\frac{12196224}{1885}\,\EuFrak{b}_{48,8}(n)  \\
+  \frac{589824}{65}\,\EuFrak{b}_{48,9}(n)   
+ \frac{10119168}{1885}\,\EuFrak{b}_{48,10}(n) 
 -\frac{140544}{65}\,\EuFrak{b}_{48,11}(n)  
-\frac{118171008}{1885}\,\EuFrak{b}_{48,12}(n)  \\
 -\frac{3032064}{65}\,\EuFrak{b}_{48,13}(n)  
 -\frac{7167744}{377}\,\EuFrak{b}_{48,14}(n) 
 -\frac{562176}{65}\,\EuFrak{b}_{48,15}(n)  \\
-\frac{32182272}{377}\,\EuFrak{b}_{48,16}(n) 
+ \frac{2313216}{13}\,\EuFrak{b}_{48,17}(n)  
+\frac{35136}{65}\,\EuFrak{b}_{48,18}(n)  
 \,\biggr)q^{n},
\label{convolSum-eqn-3_16} 
\end{multline}
\begin{multline}
(2\,L(q^{2}) - 25\,L(q^{25}))^{2}  = 529 + \sum_{n=1}^{\infty}\biggl(\, 
 \frac{810}{13}\,\sigma_{3}(n)   
+  \frac{11460}{13}\,\sigma_{3}(\frac{n}{2}) \\
- \frac{3210}{13}\,\sigma_{3}(\frac{n}{5})  
- 660\,\sigma_{3}(\frac{n}{10})  
+ \frac{1890000}{13}\,\sigma_{3}(\frac{n}{25})
 -\frac{240000}{13}\,\sigma_{3}(\frac{n}{50}) \\
  -\frac{810}{13}\,\EuFrak{b}_{50,1}(n)  
+ \frac{6714}{13}\,\EuFrak{b}_{50,2}(n)
 - 1620\,\EuFrak{b}_{50,3}(n)
 - 4230\,\EuFrak{b}_{50,4}(n)
 - \frac{178950}{13}\,\EuFrak{b}_{50,5}(n)   \\
 -  20250\,\EuFrak{b}_{50,6}(n)  
+ 810\,\EuFrak{b}_{50,7}(n)  
- 13050\,\EuFrak{b}_{50,8}(n)  
+  12420\,\EuFrak{b}_{50,9}(n)   \\
- \frac{68400}{13}\,\EuFrak{b}_{50,10}(n) 
 - 4500\,\EuFrak{b}_{50,11}(n)  
- 36000\,\EuFrak{b}_{50,12}(n)
 - 21150\,\EuFrak{b}_{50,13}(n)  \\
 + 1800\,\EuFrak{b}_{50,14}(n) 
 - 15000\,\EuFrak{b}_{50,15}(n)  
+ 20700\,\EuFrak{b}_{50,16}(n)   
+ 28800\,\EuFrak{b}_{50,17}(n)  
\,\biggr)\,q^{n}.
\label{convolSum-eqn-2_25}
\end{multline}
\begin{multline}
(L(q) - 50\, L(q^{50}))^{2} 
 = 2401 + \sum_{n=1}^{\infty}\biggl(\, 
 \frac{7590}{13}\,\sigma_{3}(n) 
 + \frac{13500}{13}\,\sigma_{3}(\frac{n}{2}) \\
 -\frac{6870}{13}\,\sigma_{3}(\frac{n}{5})  
 -\frac{23100}{13}\,\sigma_{3}(\frac{n}{10})
-  \frac{60000}{13}\,\sigma_{3}(\frac{n}{25})
+ \frac{7560000}{13}\,\sigma_{3}(\frac{n}{50})
+  \frac{38772}{13}\,\EuFrak{b}_{50,1}(n)  \\
+ \frac{639792}{13}\,\EuFrak{b}_{50,2}(n)  
- 7020\,\EuFrak{b}_{50,3}(n)  
- 20250\,\EuFrak{b}_{50,4}(n)
- \frac{721050}{13}\,\EuFrak{b}_{50,5}(n)  \\
- 116550\,\EuFrak{b}_{50,6}(n)  
- 2250\,\EuFrak{b}_{50,7}(n)
- 123750\,\EuFrak{b}_{50,8}(n)  
+ 99900\,\EuFrak{b}_{50,9}(n)
- \frac{910800}{13}\,\EuFrak{b}_{50,10}(n)  \\
- 38700\,\EuFrak{b}_{50,11}(n)
- 309600\,\EuFrak{b}_{50,12}(n)  
+ 13950\,\EuFrak{b}_{50,13}(n)    
- 88200\,\EuFrak{b}_{50,14}(n)  \\
- 129000\,\EuFrak{b}_{50,15}(n)
- 6300\,\EuFrak{b}_{50,16}(n) 
+ 28800\,\EuFrak{b}_{50,17}(n) \,\biggr)\,q^{n}.
\label{convolSum-eqn-1_50}
\end{multline}
\begin{multline}
( L(q) - 64\, L(q^{64}))^{2}   
 = 3969 + \sum_{n=1}^{\infty}\biggl(\, 
    234\,\sigma_{3}(n)
 - 18\,\sigma_{3}(\frac{n}{2})  \\     
 +  \frac{69624}{13}\,\sigma_{3}(\frac{n}{4})   
   - \frac{74304}{13}\,\sigma_{3}(\frac{n}{8}) 
- 1152\,\sigma_{3}(\frac{n}{16})   
 - 4608\,\sigma_{3}(\frac{n}{32})   
 + 958464 \,\sigma_{3}(\frac{n}{64})  \\
 + 2790\,\EuFrak{b}_{64,1}(n)  
 + 7560\,\EuFrak{b}_{64,2}(n)
 + 20160 \,\EuFrak{b}_{64,3}(n)
 + \frac{112896}{13}\,\EuFrak{b}_{64,4}(n)  
 + 96768 \,\EuFrak{b}_{64,5}(n)     \\
 + 48384 \,\EuFrak{b}_{64,6}(n)   
 + 96768 \,\EuFrak{b}_{64,7}(n)
 + 17280 \,\EuFrak{b}_{64,8}(n)
 + 73728 \,\EuFrak{b}_{64,9}(n)
+ 221184 \,\EuFrak{b}_{64,10}(n)    \\
+ 331776 \,\EuFrak{b}_{64,11}(n)
 + 64512 \,\EuFrak{b}_{64,12}(n)   
- 221184\,\EuFrak{b}_{64,13}(n)
- 276480\,\EuFrak{b}_{64,14}(n)        \\
+ 368640 \,\EuFrak{b}_{64,15}(n)  
- \frac{564480}{13}\,\EuFrak{b}_{64,16}(n)
+ 1290240\,\EuFrak{b}_{64,17}(n)
+ 110592 \,\EuFrak{b}_{64,18}(n)
\,\biggr)q^{n}, \label{convolSum-eqn-1_64}  
\end{multline}
\end{corollary}
\begin{proof} It follows immediately from 
\hyperref[convolution-lemma_a_b]{Lemma \ref*{convolution-lemma_a_b}}  
when one sets $\alpha=5$ and $\beta=9$. However, we briefly show the proof 
for $( 5\,L(q^{5}) - 9\,L(q^{9}) )^{2}$ as an example. One obtains 
\begin{align}
( 5\,L(q^{5}) - 9\,L(q^{9}) )^{2}  & =   
\sum_{\delta|45}\,x_{\delta} M(q^{\delta}) + z_{1}\,M_{\legendre{-4}{n}}(q) 
+ z_{3}\,M_{\legendre{-4}{n}}(q^{3})  
  + \sum_{j=1}^{14}\,y_{j}\,\EuFrak{B}_{45,j}(q)  \notag \\ &  
  = \sum_{\delta|45}\,x_{\delta} 
  + \sum_{i=1}^{\infty}\biggl( 
  \sum_{\delta|45}\,240\,\sigma_{3}(\frac{n}{\delta})\,x_{\delta} 
  + \legendre{-4}{n}\,\sigma_{3}(n)\,z_{1} 
  + \legendre{-4}{n}\,\sigma_{3}(\frac{n}{3})\,z_{3}  \notag \\ &
  \quad + \sum_{j=1}^{14}\,\EuFrak{b}_{45,j}(n)\,y_{j}\,\biggr)\,q^{n}. 
      \label{convolution_5_9-eqn-0}
\end{align} 
since the conductor of the Dirichlet character $\legendre{-4}{n}$ is $4$, and hence  
from \autoref{Eisenstein-gen} we have $C_{0}=0$.   
Now when we equate the right hand side of 
\autoref{convolution_5_9-eqn-0} with that of 
\autoref{evalConvolClass-eqn-11}, and when we take the 
coefficients of $q^{n}$ for which $1\leq n\leq 15$ and 
$n=17,19,21,23,25,27,45$ for example, we 
obtain a system of linear equations with a unique solution. 
Hence, we obtain the stated result.
\end{proof} 
 
Now we state and prove our main result of this subsection. 
\begin{corollary} \label{convolSum-theor-w_45_50}
Let $n$ be a positive integer. Then 
\begin{align}
 W_{(5,9)}(n)  = & 
   \frac{1}{5616}\sigma_{3}(n) 
   + \frac{433}{398736}\sigma_{3}(\frac{n}{3}) 
    + \frac{25}{5616}\sigma_{3}(\frac{n}{5}) 
    + \frac{13}{568}\sigma_{3}(\frac{n}{9})        \notag \\ &
    + \frac{42415}{398736}\sigma_{3}(\frac{n}{15}) 
    - \frac{100}{71}\sigma_{3}(\frac{n}{45}) 
    + (\frac{1}{24}-\frac{1}{36}n)\sigma(\frac{n}{5}) \notag \\ &
   + (\frac{1}{24}-\frac{1}{20}n)\sigma(\frac{n}{9})    
   + \frac{403}{1456920} \,\EuFrak{b}_{45,1}(n) 
   - \frac{3323}{971280} \,\EuFrak{b}_{45,2}(n)      \notag \\ &
  - \frac{448373}{37879920} \,\EuFrak{b}_{45,3}(n) 
  - \frac{15889}{1456920} \,\EuFrak{b}_{45,4}(n) 
   + \frac{6406}{182115} \,\EuFrak{b}_{45,5}(n)    \notag \\ &
   - \frac{5197}{291384} \,\EuFrak{b}_{45,6}(n)   
  - \frac{3331}{40470} \,\EuFrak{b}_{45,7}(n) 
  - \frac{257}{2556} \,\EuFrak{b}_{45,8}(n)       \notag \\ &
   - \frac{52159}{1456920} \,\EuFrak{b}_{45,9}(n) 
   - \frac{35}{426} \,\EuFrak{b}_{45,10}(n)         
  + \frac{3331}{16188} \,\EuFrak{b}_{45,11}(n)      \notag \\ &
  - \frac{5741}{12626640} \,\EuFrak{b}_{45,12}(n) 
  - \frac{49}{291384} \,\EuFrak{b}_{45,13}(n). 
\label{convolutionSum-w_5_9}
\end{align}
\begin{align}
 W_{(1,45)}(n)  = &
   \frac{217}{1872}\sigma_{3}(n) 
   + \frac{433}{398736}\sigma_{3}(\frac{n}{3}) 
    - \frac{625}{5616}\sigma_{3}(\frac{n}{5}) 
    - \frac{25}{71}\sigma_{3}(\frac{n}{9})         \notag \\ &
    + \frac{42415}{398736}\sigma_{3}(\frac{n}{15}) 
    - \frac{587}{568}\sigma_{3}(\frac{n}{45}) 
    + (\frac{1}{24}-\frac{1}{36}n)\sigma(\frac{n}{5}) \notag \\ &
   + (\frac{1}{24}-\frac{1}{20}n)\sigma(\frac{n}{9})    
   + \frac{403}{1456920} \,\EuFrak{b}_{45,1}(n) 
   - \frac{3323}{971280} \,\EuFrak{b}_{45,2}(n)          \notag \\ &
  - \frac{448373}{37879920} \,\EuFrak{b}_{45,3}(n) 
  - \frac{15889}{1456920} \,\EuFrak{b}_{45,4}(n) 
   + \frac{6406}{182115} \,\EuFrak{b}_{45,5}(n)      \notag \\ &
   - \frac{5197}{291384} \,\EuFrak{b}_{45,6}(n)     
  - \frac{3331}{40470} \,\EuFrak{b}_{45,7}(n) 
  - \frac{257}{2556} \,\EuFrak{b}_{45,8}(n)        \notag \\ &
   - \frac{52159}{1456920} \,\EuFrak{b}_{45,9}(n)     
   - \frac{35}{426} \,\EuFrak{b}_{45,10}(n)       
  + \frac{3331}{16188} \,\EuFrak{b}_{45,11}(n)     \notag \\ &
  - \frac{5741}{12626640} \,\EuFrak{b}_{45,12}(n) 
  - \frac{49}{291384} \,\EuFrak{b}_{45,13}(n). 
\label{convolutionSum-w_1_45}
\end{align}
\begin{align}
W_{(1,48)}(n)   = & 
\frac{1}{7680}\,\sigma_{3}(n)
+ \frac{189}{33280}\,\sigma_{3}(\frac{n}{2})
+\frac{3}{2560}\,\sigma_{3}(\frac{n}{3})
+ \frac{673}{8320}\,\sigma_{3}(\frac{n}{4}) \notag \\ &
 -\frac{59}{33280}\,\sigma_{3}(\frac{n}{6})
 -\frac{163}{2080}\,\sigma_{3}(\frac{n}{8})
 -\frac{543}{8320}\,\sigma_{3}(\frac{n}{12})
 + \frac{1}{30}\,\sigma_{3}(\frac{n}{16})  
+ \frac{293}{2080}\,\sigma_{3}(\frac{n}{24})  \notag \\ &
+ \frac{3}{10}\,\sigma_{3}(\frac{n}{48})
 + (\frac{1}{24}-\frac{1}{192}n)\sigma(n)         
 + (\frac{1}{24}-\frac{1}{4}n)\sigma(\frac{n}{48})  
-\frac{359}{33280}\,\EuFrak{b}_{48,1}(n)    \notag \\ &
-\frac{1481}{8320}\,\EuFrak{b}_{48,2}(n) 
-\frac{5063}{33280}\,\EuFrak{b}_{48,3}(n) 
 -\frac{383}{2080}\,\EuFrak{b}_{48,4}(n) 
 -\frac{113}{2080}\,\EuFrak{b}_{48,5}(n) \notag \\ &
-\frac{8633}{8320}\,\EuFrak{b}_{48,6}(n) 
 -\frac{59}{240}\,\EuFrak{b}_{48,7}(n) 
-\frac{1569}{1040}\,\EuFrak{b}_{48,8}(n) 
 -\frac{479}{390}\,\EuFrak{b}_{48,9}(n)   \notag \\ &
 -\frac{57}{65}\,\EuFrak{b}_{48,10}(n) 
 -\frac{619}{1560}\,\EuFrak{b}_{48,11}(n) 
-\frac{4093}{1040}\,\EuFrak{b}_{48,12}(n) 
+ \frac{212}{195}\,\EuFrak{b}_{48,13}(n)   \notag \\ &
+ \frac{245}{104}\,\EuFrak{b}_{48,14}(n)   
 -\frac{17}{195}\,\EuFrak{b}_{48,15}(n) 
 -\frac{132}{13}\,\EuFrak{b}_{48,16}(n) 
 -\frac{56}{39}\,\EuFrak{b}_{48,17}(n)      \notag \\ &
 -\frac{161}{6240}\,\EuFrak{b}_{48,18}(n), 
\label{convolSum-theor-w_1_48}  
\end{align} 
\begin{align}
W_{(3,16)}(n)  = & 
 \frac{1}{7680}\,\sigma_{3}(n)
-\frac{399}{965120}\,\sigma_{3}(\frac{n}{2}) 
+\frac{3}{2560}\,\sigma_{3}(\frac{n}{3}) 
-\frac{2533}{241280}\,\sigma_{3}(\frac{n}{4}) \notag \\ &
 +\frac{4169}{965120}\,\sigma_{3}(\frac{n}{6}) 
+\frac{1153}{60320}\,\sigma_{3}(\frac{n}{8}) 
 +\frac{6303}{241280}\,\sigma_{3}(\frac{n}{12}) 
 +\frac{1}{30}\,\sigma_{3}(\frac{n}{16})      \notag \\ &
+\frac{2617}{60320}\,\sigma_{3}(\frac{n}{24}) 
 +\frac{3}{10}\,\sigma_{3}(\frac{n}{48}) 
 + (\frac{1}{24}-\frac{1}{64}n)\sigma(\frac{n}{3}) \notag \\ &
 + (\frac{1}{24}-\frac{1}{12}n)\sigma(\frac{n}{16})    
 +\frac{321}{33280}\, \EuFrak{b}_{48,1}(n) 
-\frac{7259}{241280}\, \EuFrak{b}_{48,2}(n) \notag \\ &
+\frac{1297}{33280}\, \EuFrak{b}_{48,3}(n) 
+\frac{2643}{60320}\, \EuFrak{b}_{48,4}(n) 
 + \frac{67}{2080}\, \EuFrak{b}_{48,5}(n) 
 -\frac{47}{241280}\, \EuFrak{b}_{48,6}(n) \notag \\ &
 -\frac{79}{240}\, \EuFrak{b}_{48,7}(n) 
+\frac{3529}{30160}\, \EuFrak{b}_{48,8}(n) 
- \frac{32}{195}\, \EuFrak{b}_{48,9}(n) 
- \frac{183}{1885}\, \EuFrak{b}_{48,10}(n)   \notag \\ &
 + \frac{61}{1560}\, \EuFrak{b}_{48,11}(n) 
+ \frac{34193}{30160}\, \EuFrak{b}_{48,12}(n) 
+ \frac{329}{390}\, \EuFrak{b}_{48,13}(n) 
+ \frac{1037}{3016}\, \EuFrak{b}_{48,14}(n)    \notag \\ &
+ \frac{61}{390}\, \EuFrak{b}_{48,15}(n) 
+ \frac{582}{377}\, \EuFrak{b}_{48,16}(n) 
- \frac{251}{78}\, \EuFrak{b}_{48,17}(n) 
- \frac{61}{6240}\, \EuFrak{b}_{48,18}(n),
\label{convolSum-theor-w_3_16}  
 \end{align}
\begin{align}
 W_{(2,25)}(n)  =  & 
 - \frac{9}{8320}\,\sigma_{3}(n)
+ \frac{17}{12480}\,\sigma_{3}(\frac{n}{2}) 
+\frac{107}{24960}\,\sigma_{3}(\frac{n}{5}) 
+ \frac{11}{960}\,\sigma_{3}(\frac{n}{10}) \notag \\ &
 + \frac{25}{312}\,\sigma_{3}(\frac{n}{25}) 
+ \frac{25}{78}\,\sigma_{3}(\frac{n}{50}) 
 + (\frac{1}{24}-\frac{1}{100}n)\sigma(\frac{n}{2}) 
 + (\frac{1}{24}-\frac{1}{8}n)\sigma(\frac{n}{25})  \notag \\ & 
 +\frac{9}{8320}\, \EuFrak{b}_{50,1}(n) 
-\frac{373}{41600}\, \EuFrak{b}_{50,2}(n) 
+\frac{9}{320}\, \EuFrak{b}_{50,3}(n) 
+\frac{47}{640}\, \EuFrak{b}_{50,4}(n)       \notag \\ &
 + \frac{1193}{4992}\, \EuFrak{b}_{50,5}(n) 
 + \frac{45}{128}\, \EuFrak{b}_{50,6}(n) 
 -\frac{9}{640}\, \EuFrak{b}_{50,7}(n) 
+\frac{29}{128}\, \EuFrak{b}_{50,8}(n)     \notag \\ &
 -\frac{69}{320}\, \EuFrak{b}_{50,9}(n) 
 +\frac{19}{208}\, \EuFrak{b}_{50,10}(n) 
 + \frac{5}{64}\, \EuFrak{b}_{50,11}(n) 
+\frac{5}{8}\, \EuFrak{b}_{50,12}(n)        \notag \\ &
+ \frac{47}{128}\, \EuFrak{b}_{50,13}(n)   
- \frac{1}{32}\, \EuFrak{b}_{50,14}(n) 
+ \frac{25}{96}\, \EuFrak{b}_{50,15}(n) 
- \frac{23}{64}\, \EuFrak{b}_{50,16}(n)  \notag \\ &
+ \frac{1}{2}\, \EuFrak{b}_{50,17}(n) 
\label{convolutionSum-w_2_25}
\end{align}
\begin{align}
 W_{(1,50)}(n)  = &
- \frac{149}{24960}\,\sigma_{3}(n)
- \frac{15}{832}\,\sigma_{3}(\frac{n}{2})
+ \frac{229}{24960}\,\sigma_{3}(\frac{n}{5})
+ \frac{77}{2496}\,\sigma_{3}(\frac{n}{10})    \notag \\ &
+ \frac{25}{312}\,\sigma_{3}(\frac{n}{25})
+ \frac{25}{78}\,\sigma_{3}(\frac{n}{50})
 + (\frac{1}{24}-\frac{1}{200}n)\sigma(n)         
 + (\frac{1}{24}-\frac{1}{4}n)\sigma(\frac{n}{50})  \notag \\ &
- \frac{1277}{41600}\,\EuFrak{b}_{50,1}(n) 
- \frac{243}{1664}\,\EuFrak{b}_{50,2}(n) 
+ \frac{39}{320}\,\EuFrak{b}_{50,3}(n) 
+ \frac{45}{128}\,\EuFrak{b}_{50,4}(n)       \notag \\ &
+ \frac{4807}{4992}\,\EuFrak{b}_{50,5}(n)    
+ \frac{259}{128}\,\EuFrak{b}_{50,6}(n) 
+ \frac{5}{128}\,\EuFrak{b}_{50,7}(n) 
+ \frac{275}{128}\,\EuFrak{b}_{50,8}(n)      \notag \\ &
- \frac{111}{64}\,\EuFrak{b}_{50,9}(n) 
+ \frac{253}{208}\,\EuFrak{b}_{50,10}(n)    
+ \frac{43}{64}\,\EuFrak{b}_{50,11}(n) 
+ \frac{43}{8}\,\EuFrak{b}_{50,12}(n)        \notag \\ &
- \frac{31}{128}\,\EuFrak{b}_{50,13}(n) 
+ \frac{49}{32}\,\EuFrak{b}_{50,14}(n) 
+ \frac{215}{96}\,\EuFrak{b}_{50,15}(n)      
+ \frac{7}{64}\,\EuFrak{b}_{50,16}(n)        \notag \\ &
- \frac{1}{2}\,\EuFrak{b}_{50,17}(n). 
\label{convolutionSum-w_1_50}
\end{align}
\begin{align}
W_{(1,64)}(n)  = & 
     \frac{1}{12288}\,\sigma_{3}(n)
  + \frac{1}{4096}\,\sigma_{3}(\frac{n}{2})    
     - \frac{967}{13312} \,\sigma_{3}(\frac{n}{4}) 
   +  \frac{129}{1664}\,\sigma_{3}(\frac{n}{8})     \notag \\ & 
  + \frac{1}{64}\,\sigma_{3}(\frac{n}{16})   
  + \frac{1}{16}\,\sigma_{3}(\frac{n}{32})
  + \frac{1}{3}\,\sigma_{3}(\frac{n}{64})
 + (\frac{1}{24}-\frac{1}{208}n)\sigma(n)  \notag \\ & 
 + (\frac{1}{24}-\frac{1}{4}n)\sigma(\frac{n}{64})   
 - \frac{155}{4096} \,\EuFrak{b}_{64,1}(n) 
 - \frac{105}{1024}\,\EuFrak{b}_{64,2}(n)         
  - \frac{35}{128}\,\EuFrak{b}_{64,3}(n)    \notag \\ &        
  - \frac{49}{416} \,\EuFrak{b}_{64,4}(n)      
  - \frac{21}{16} \,\EuFrak{b}_{64,5}(n)        
  - \frac{21}{32} \,\EuFrak{b}_{64,6}(n)        
  - \frac{21}{16}\,\EuFrak{b}_{64,7}(n)    \notag \\ &        
  - \frac{15}{64} \,\EuFrak{b}_{64,8}(n)         
  - \,\EuFrak{b}_{64,9}(n)        
  -3 \,\EuFrak{b}_{64,10}(n)        
 - \frac{9}{2}\,\EuFrak{b}_{64,11}(n)        
 - \frac{7}{8}\,\EuFrak{b}_{64,12}(n)   \notag \\ &      
  + 3\,\EuFrak{b}_{64,13}(n)         
 + \frac{15}{4}\,\EuFrak{b}_{64,14}(n)        
  -5 \,\EuFrak{b}_{64,15}(n)             
  + \frac{245}{416}\,\EuFrak{b}_{64,16}(n)   \notag \\ &     
 - \frac{35}{2} \,\EuFrak{b}_{64,17}(n)      
 - \frac{3}{2}\,\EuFrak{b}_{64,18}(n). 
\label{convolSum-theor-w_1_64}  
\end{align} 
\end{corollary}
 
\begin{proof} 
It follows immediately from \autoref{convolution_a_b} when we set 
$(\alpha,\beta)=(1,16)$, $(1,25)$, $(5,9)$, $(1,45)$, $(2,25)$, $(1,50)$, $(1,64)$. 
\end{proof}


\section{Number of Representations of a positive Integer  $n$  by the Octonary 
Quadratic Form \autoref{introduction-eq-1}
}
\label{representations_48_64-ab}

We apply the convolution sums $W_{(1,48)}(n)$, 
$W_{(3,16)}(n)$, $W_{(1,64)}(n)$ and 
other known evaluated convolution sums to determine 
explicit formulae for the number of representations of a positive
integer $n$  by the octonary quadratic form \autoref{introduction-eq-1}.

Since $64=2^{6}$ and $48=2^{4}\cdot 3$ it follows from 
\autoref{representations_a_b-eqn-1} that   
$\Omega_{4}=\{(1,16)\}\cup\{(3,4), (1,12)\}$.

The following result is then deduced.
\begin{corollary} \label{representations-thrm_3_4}
Let $n\in\mathbb{N}$ and $(a,b)=(1,12),(1,16),(3,4)$. Then  
\begin{align*}
N_{(1,12)}(n)  = & 
8\sigma(n) - 32\sigma(\frac{n}{4}) + 8\sigma(\frac{n}{12}) -
32\sigma(\frac{n}{48}) + 64\, W_{(1,12)}(n)  \\ & 
+ 1024\, W_{(1,12)}(\frac{n}{4}) 
 - 256\, \biggl( W_{(1,3)}(\frac{n}{4}) + W_{(1,48)}(n) \biggr). \\ 
N_{(1,16)}(n)  = & 
8\sigma(n) - 32\sigma(\frac{n}{4}) + 8\sigma(\frac{n}{16}) -
32\sigma(\frac{n}{64}) + 64\, W_{(1,16)}(n) \\ & 
+ 1024\, W_{(1,16)}(\frac{n}{4}) 
 - 256\, \biggl( W_{(1,4)}(\frac{n}{4}) + W_{(1,64)}(n) \biggr). \\ 
N_{(3,4)}(n)  = & 
8\sigma(\frac{n}{3}) - 32\sigma(\frac{n}{12}) + 8\sigma(\frac{n}{4}) -
32\sigma(\frac{n}{16}) + 64\, W_{(3,4)}(n) \\ &
 + 1024\, W_{(3,4)}(\frac{n}{4}) - 256\, \biggl( W_{(1,3)}(\frac{n}{4}) 
 + W_{(3,16)}(n) \biggr). 
\end{align*}
\end{corollary}
\begin{proof} 
These identities follow immediately from  \autoref{representations-theor_a_b}. 
We can make use of the results obtained by \aA\ et al. \cite{alaca_alaca_williams2006} 
and \jgH\ et al.\ \cite[Thrm 3, p.\ 20]{huardetal}, and 
\autoref{convolSum-theor-w_3_16} and \autoref{convolSum-theor-w_1_48} to 
simplify for example $N_{(1,12)}(n)$ and $N_{(3,4)}(n)$.
\end{proof}

\section{Number of Representations of a Positive Integer $n$  by 
the Octonary Quadratic Form \autoref{introduction-eq-2}
} 
\label{representations_5_9}

We make use of the convolution sums $W_{(5,9)}(n)$, $W_{(1,45)}(n)$, 
$W_{(3,16)}(n)$, $W_{(1,48)}(n)$ and 
other well-known convolution sums to determine 
explicit formulae for the number of representations of a positive
integer $n$  by the octonary quadratic form \autoref{introduction-eq-2}.

Since $45=3^{2}\cdot 5$ and $48=2^{4}\cdot 3$ it follows from 
\autoref{representations_c_d-eqn-1} that   
$\Omega_{3}=\{(3,5), (1,15)\}\cup\{ (1,16)\}$. 

We revisit the evaluation of the convolution sums for $\alpha\beta=5, 15$ 
using modular forms.  
The result for $\alpha\beta=5$ was obtained by \mL\ and \ksW\ \cite{lemire_williams}, 
and \sC\ and \pcT\ \cite{cooper_toh}; that for $\alpha\beta=15$ was achieved by 
\bR\ and \bS\ \cite{ramakrishnan_sahu} when using 
a basis which contains one cusp form of weight $2$. We note that 
$\alpha\beta=5, 15$ belong to the class of positive integers $\alpha\beta$ 
discussed by \eN\ \cite{ntienjem2016c}. 
Therefore, it suffices to determine a basis of the space of cusp forms for 
$\Gamma_{0}(\alpha\beta))$ and apply \cite[Thrm 3.4]{ntienjem2016c}. 
Because of \autoref{bas-s-eqn-5_45},  
$\EuFrak{B}_{45,2}(q)$, 
$\EuFrak{B}_{45,3}(q)$
and in addition 
\begin{gather*}
\EuFrak{B}'_{45,1}(q)=\eta^{4}(z)\eta^{4}(5z)=\underset{n
\geq 1}{\sum}\EuFrak{b}'_{45,1}(n)q^{n}, \\
\EuFrak{B}'_{45,4}(q)=\frac{\eta^{3}(z)\eta(3z)\eta^{7}(15z)}{\eta^{3}(5z)}=\underset{n
\geq 1}{\sum}\EuFrak{b}'_{45,4}(n)q^{n},
\end{gather*}
are basis elements of $\S_{4}(\Gamma_{0}(15))$.  
We note that $\EuFrak{B}'_{45,1}(q)$ is the basis element of $\S_{4}(\Gamma_{0}(5))$.  
\begin{theorem} \label{lemire_williams-cooper_toh}
Let $n$ be a positive integer. Then 
\begin{align}  
W_{(1,5)}(n) = & 
  \frac{5}{312}\sigma_{3}(n) 
  + \frac{125}{132}\sigma_{3}(\frac{n}{5})
  + (\frac{1}{24}-\frac{1}{20}n)\sigma(n)  
  + (\frac{1}{24}-\frac{1}{4}n)\sigma(\frac{n}{5})   \notag \\  & 
   - \frac{1}{130}\EuFrak{b}'_{45,1}(n) ,\\
W_{(3,5)}(n)  =  & \frac{1}{390}\sigma_{3}(n) 
  + \frac{7}{520}\sigma_{3}(\frac{n}{3}) 
  - \frac{175}{312}\sigma_{3}(\frac{n}{5})  
  + \frac{25}{26}\sigma_{3}(\frac{n}{15})         \notag \\ & 
  + (\frac{1}{24}-\frac{n}{20})\sigma(\frac{n}{3}) 
  + (\frac{1}{24}-\frac{n}{12})\sigma(\frac{n}{5})  
  - \frac{1}{390}\EuFrak{b}'_{45,1}(n)  
  - \frac{1}{30}\EuFrak{b}_{45,2}(n)     \notag \\ &    
  - \frac{1}{26}\EuFrak{b}_{45,3}(n)    
  - \frac{1}{5}\EuFrak{b}'_{45,4}(n)    \\   
W_{(1,15)}(n)  =  & \frac{1}{1560}\sigma_{3}(n) 
  + \frac{1}{65}\sigma_{3}(\frac{n}{3}) 
  + \frac{25}{39}\sigma_{3}(\frac{n}{5})  
  - \frac{25}{104}\sigma_{3}(\frac{n}{15}) \notag \\ & 
  + (\frac{1}{24}-\frac{n}{60})\sigma(n) 
  + (\frac{1}{24}-\frac{n}{4})\sigma(\frac{n}{15})  
  - \frac{1}{39}\EuFrak{b}'_{45,1}(n)     
  - \frac{2}{15}\EuFrak{b}_{45,2}(n)     \notag \\ &      
  - \frac{14}{65}\EuFrak{b}_{45,3}(n)                 
  + \frac{1}{5}\EuFrak{b}'_{45,4}(n).      
\end{align}
\end{theorem}
We make use of these results to deduce the following. 
\begin{corollary} \label{representations-coro-5_9-1}
Let $n\in\mathbb{N}$ and $c,d)=(1,15),(3,5), (1,16)$. Then  
\begin{align*}
R_{(1,15)}(n) =  & 12\sigma(n) - 36\sigma(\frac{n}{3}) + 12\sigma(\frac{n}{15}) -
36\sigma(\frac{n}{45}) + 144 \, W_{(1,15)}(n) \\ &
   + 1296\, W_{(1,15)}(\frac{n}{3}) - 432\, \biggl(\, W_{(1,5)}(\frac{n}{3}) + 
   W_{(1,45)}(n)\, \biggr). 
\end{align*}
\begin{align*}
R_{(3,5)}(n) = & 12\sigma(\frac{n}{3}) - 36\sigma(\frac{n}{9}) + 12\sigma(\frac{n}{5}) -
36\sigma(\frac{n}{15}) + 144\, W_{(3,5)}(n) \\ & 
   + 1296\, W_{(3,5)}(\frac{n}{3}) - 432\, \biggl(\, W_{(1,5)}(\frac{n}{3}) + 
   W_{(5,9)}(n)\, \biggr). 
\end{align*}
\begin{align*}
R_{(1,16)}(n) =  & 
12\sigma(n) - 36\sigma(\frac{n}{3}) + 12\sigma(\frac{n}{16}) -
36\sigma(\frac{n}{48}) + 144\, W_{(1,16)}(n)  \\ & 
   + 1296\, W_{(1,16)}(\frac{n}{3}) 
   - 432\, \biggl( W_{(3,16)}(n) +  W_{(1,48)}(n) \biggr). 
\end{align*}
\end{corollary}
\begin{proof} 
It follows immediately from  
\autoref{representations-theor-c_d}. 
We can make use of \autoref{lemire_williams-cooper_toh}, 
\autoref{convolutionSum-w_5_9} and \autoref{convolutionSum-w_1_45} to 
simplify $R_{(1,15)}(n)$ and $R_{(3,5)}(n)$ for example.
\end{proof}

\section{Revisited evaluation of the convolution sums for $\alpha\beta=9,16,18,25,36$}
\label{Revisited-Evaluations}

We revisit the evaluation of the convolution sums for $W_{(1,9)}(n)$, 
$W_{(1,16)}(n)$, $W_{(1,18)}(n)$, $W_{(2,9)}(n)$, $W_{(1,25)}(n)$, $W_{(1,36)}(n)$  
and $W_{(4,9)}(n)$ obtained by \ksW\ \cite{williams2005}, \aA\ et al.\ 
\cite{alaca_alaca_williams2008,alaca_alaca_williams2007},  
\exwX\ et al.\ \cite{xiaetal2014} and \dY\ \cite{ye2015}, respectively. 
These convolution sums have been evaluated using a different technique. 

Due to \autoref{bas-s-eqn-9_45}, using $\EuFrak{B}_{45,1}(q)$ as basis element of 
$\S_{4}(\Gamma_{0}(9))$ and applying the same primitive Dirichlet character as for 
$\E_{4}(\Gamma_{0}(45))$, one easily replicates the result for the 
convolution sum $W_{(1,9}(n)$ obtained by \ksW\ \cite{williams2005}. 

Observe that 
\begin{align*}   
\text{dim}(\E_{4}(\Gamma_{0}(16)))= 6, \quad 
\text{dim}(\S_{4}(\Gamma_{0}(16)))= 3,  \\
\text{dim}(\E_{4}(\Gamma_{0}(25)))=6, \quad    
\text{dim}(\S_{4}(\Gamma_{0}(25)))=5.
\end{align*}
These convolution sums are improved using 
our method since we apply the right number of basis elements of the space of cusp 
forms corresponding to level $16$, $18$ and $25$. 
In case of the evaluation of $W_{(1,16)}(n)$, we will use 
$\EuFrak{B}'_{64,3}(q) = \frac{\eta^{6}(4z)\eta^{4}(16z)}{\eta^{2}(8z)}=
\underset{n\geq 1}{\sum}\EuFrak{b}'_{64,3}(n)q^{n}$ instead 
of $\EuFrak{B}_{64,3}(q)$ given in \autoref{convolutionSums-1_64-table}; 
the primitive Dirichlet character \autoref{base-3_16-kronecker} is applicable.
For the evaluation of $W_{(1,25)}(n)$ we will use 
$\EuFrak{B}'_{50,2}(q) = \eta^{3}(z)\eta^{4}(5z)\eta(25z)=
\underset{n\geq 1}{\sum}\EuFrak{b}'_{50,2}(n)q^{n}$ instead 
of $\EuFrak{B}_{50,2}(q)$ given in \autoref{convolutionSums-2_25-table}; 
we apply the primitive Dirichlet character \autoref{base-3_16-kronecker}.

Since 
\begin{align}
\M_{4}(\Gamma_{0}(6)) \subset \M_{4}(\Gamma_{0}(12)) \subset \M_{4}(\Gamma_{0}(36)) 
 \label{bas-s-eqn-6_36} \\
\M_{4}(\Gamma_{0}(9)) \subset \M_{4}(\Gamma_{0}(18)) \subset \M_{4}(\Gamma_{0}(36)) 
\label{bas-s-eqn-9_36} 
\end{align} 
it suffices to consider the basis of $\S_{4}(\Gamma_{0}(36))$, whose table of 
the exponent of the $\eta$-quotients is given in \autoref{convolutionSums-1_36-table}.
Note that 
\begin{align*}   
\text{dim}(\E_{4}(\Gamma_{0}(18)))= 8, \quad 
\text{dim}(\S_{4}(\Gamma_{0}(18)))= 5,  \\
\text{dim}(\E_{4}(\Gamma_{0}(36)))=12, \quad    
\text{dim}(\S_{4}(\Gamma_{0}(36)))=12.
\end{align*}
The primitive Dirichlet character  
\hyperref[base-5_9-kronecker]{Equation \ref*{base-5_9-kronecker}} 
is applicable in case of $\E_{4}(\Gamma_{0}(18))$ and $\E_{4}(\Gamma_{0}(36))$. 

\begin{corollary} \label{lema_18_36}
It holds that 
\begin{multline}
( L(q) - 9\, L(q^{9}))^{2} 
 = 64 + \sum_{n=1}^{\infty}\biggl(\,
  192\,\sigma_{3}(n)
 - 384\,\sigma_{3}(\frac{n}{3})
+ 15552\,\sigma_{3}(\frac{n}{9})  \\
+  192\,\EuFrak{b}_{45,1}(n)
\,\biggr)q^{n}.
\label{convolSum-eqn-1_9}
\end{multline}
\begin{multline}
( L(q) - 16\, L(q^{16}))^{2} 
 = 225 + \sum_{n=1}^{\infty}\biggl(\,
  216\,\sigma_{3}(n) 
  - 72\,\sigma_{3}(\frac{n}{2}) 
 - 288\,\sigma_{3}(\frac{n}{4}) \\
- 1152\,\sigma_{3}(\frac{n}{8}) 
+ 55296\,\sigma_{3}(\frac{n}{16}) 
 + 504\,\EuFrak{b}_{64,1}(n) 
 + 864\,\EuFrak{b}_{64,2}(n) 
+ 2304\,\EuFrak{b}'_{64,3}(n) 
\,\biggr)q^{n}.
\label{convolSum-eqn-1_16}
\end{multline}
\begin{multline}
( L(q) - 18\, L(q^{18}))^{2} 
 = 289 + \sum_{n=1}^{\infty}\biggl(\,
   \frac{1104}{5}\,\sigma_{3}(n) 
  - \frac{384}{5}\,\sigma_{3}(\frac{n}{2}) 
 - \frac{768}{5}\,\sigma_{3}(\frac{n}{3}) \\
- \frac{3072}{5}\,\sigma_{3}(\frac{n}{6}) 
- \frac{7776}{5}\,\sigma_{3}(\frac{n}{9}) 
+ \frac{357696}{5}\,\sigma_{3}(\frac{n}{18}) 
+  \frac{2976}{5}\,\EuFrak{b}_{36,1}(n) 
+  \frac{8544}{5}\,\EuFrak{b}_{36,2}(n)   \\
+ \frac{17952}{5}\,\EuFrak{b}_{36,3}(n) 
+ \frac{53376}{5}\,\EuFrak{b}'_{36,4}(n) 
- \frac{52992}{5}\,\EuFrak{b}_{36,5}(n) 
\,\biggr)q^{n},
\label{convolSum-eqn-1_18}
\end{multline}
\begin{multline}
(2\,L(q) - 9\, L(q^{9}))^{2} 
 = 49 + \sum_{n=1}^{\infty}\biggl(\,
 - \frac{96}{5}\,\sigma_{3}(n)
 + \frac{4416}{5}\,\sigma_{3}(\frac{n}{2}) 
 - \frac{768}{5}\,\sigma_{3}(\frac{n}{3})    \\
- \frac{3072}{5}\,\sigma_{3}(\frac{n}{6}) 
+ \frac{89424}{5}\,\sigma_{3}(\frac{n}{9}) 
- \frac{31104}{5}\,\sigma_{3}(\frac{n}{18}) 
 +  \frac{96}{5}\,\EuFrak{b}_{36,1}(n)
 - \frac{96}{5}\,\EuFrak{b}_{36,2}(n)    \\
 + \frac{3552}{5}\,\EuFrak{b}_{36,3}(n) 
- \frac{4224}{5}\,\EuFrak{b}'_{36,4}(n)       
+  \frac{16128}{5}\,\EuFrak{b}_{36,5}(n)  
\,\biggr)q^{n},
\label{convolSum-eqn-2_9}
\end{multline}
\begin{multline}
( L(q) - 25\, L(q^{25}))^{2} 
 = 576 + \sum_{n=1}^{\infty}\biggl(\,
 \frac{2880}{13}\,\sigma_{3}(n)    
 - \frac{5760}{13}\,\sigma_{3}(\frac{n}{5})      \\
 + \frac{1800000}{13}\,\sigma_{3}(\frac{n}{25})  
 + \frac{12096}{13}\,\EuFrak{b}_{50,1}(n)    
+  5760\,\EuFrak{b}'_{50,2}(n)  
+  17280\,\EuFrak{b}_{50,3}(n)      \\
+  28800\,\EuFrak{b}_{50,4}(n)    
+ \frac{302400}{13}\,\EuFrak{b}_{50,5}(n)   
\,\biggr)q^{n}.
\label{convolSum-eqn-1_25}
\end{multline} 
\begin{multline}
( L(q) - 36\, L(q^{36}))^{2} 
 =  1225 + \sum_{n=1}^{\infty}\biggl(\,
   \frac{1152}{5}\,\sigma_{3}(n) 
   - \frac{1008}{5}\,\sigma_{3}(\frac{n}{2})  
   - \frac{384}{5}\,\sigma_{3}(\frac{n}{3})     \\
   + \frac{13056}{5}\,\sigma_{3}(\frac{n}{4}) 
   - \frac{288}{5}\,\sigma_{3}(\frac{n}{6}) 
   - \frac{3888}{5}\,\sigma_{3}(\frac{n}{9}) 
  - \frac{19968}{5}\,\sigma_{3}(\frac{n}{12})   
  - \frac{11664}{5}\,\sigma_{3}(\frac{n}{18})     \\
  + \frac{1492992}{5}\,\sigma_{3}(\frac{n}{36}) 
  +  \frac{7248}{5}\,\EuFrak{b}_{36,1}(n)   
   + 3744\,\EuFrak{b}_{36,2}(n)    
   + \frac{19008}{5}\,\EuFrak{b}_{36,3}(n)    
   + \frac{77664}{5}\,\EuFrak{b}_{36,4}(n)       \\
   + \frac{14688}{5}\,\EuFrak{b}_{36,5}(n) 
   + 16416\,\EuFrak{b}_{36,6}(n)   
   + \frac{80064}{5}\,\EuFrak{b}_{36,7}(n)    
   + \frac{84672}{5}\,\EuFrak{b}_{36,8}(n)        \\
   - 12960\,\EuFrak{b}_{36,9}(n)  
   + \frac{90624}{5}\,\EuFrak{b}_{36,10}(n)  
   + 5184\,\EuFrak{b}_{36,11}(n)   
   + 2592\,\EuFrak{b}_{36,12}(n)    
\,\biggr)q^{n}.
\label{convolSum-eqn-1_36}
\end{multline}
\begin{multline}
(4\,L(q) - 9\, L(q^{9}))^{2} 
 =  25 + \sum_{n=1}^{\infty}\biggl(\,
   \frac{1152}{5}\,\sigma_{3}(n)      
 - \frac{1008}{5}\,\sigma_{3}(\frac{n}{2})  
  - \frac{384}{5}\,\sigma_{3}(\frac{n}{3})    \\
+ \frac{13056}{5}\,\sigma_{3}(\frac{n}{4})  
+ \frac{80928}{5}\,\sigma_{3}(\frac{n}{6})  
 - \frac{3888}{5}\,\sigma_{3}(\frac{n}{9})  
- \frac{913344}{5}\,\sigma_{3}(\frac{n}{12})  
 - 505872\,\sigma_{3}(\frac{n}{18})            \\
+ \frac{29187648}{5}\,\sigma_{3}(\frac{n}{36})
 + \frac{7248}{5}\,\EuFrak{b}_{36,1}(n)   
 + 3744 \,\EuFrak{b}_{36,2}(n)          
+ \frac{19008}{5}\,\EuFrak{b}_{36,3}(n)   
+ \frac{77664}{5}\,\EuFrak{b}_{36,4}(n)      \\
+ \frac{14688}{5}\,\EuFrak{b}_{36,5}(n)           
+ \frac{864}{5}\,\EuFrak{b}_{36,6}(n)             
+ \frac{80064}{5}\,\EuFrak{b}_{36,7}(n)  
+ \frac{84672}{5}\,\EuFrak{b}_{36,8}(n)    
- 12960\,\EuFrak{b}_{36,9}(n)               \\   
+ \frac{90624}{5}\,\EuFrak{b}_{36,10}(n)      
+ 5184\,\EuFrak{b}_{36,11}(n)     
+ 2592\,\EuFrak{b}_{36,12}(n)   
\,\biggr)q^{n}.
\label{convolSum-eqn-4_9}
\end{multline}
\end{corollary}
\begin{proof} Similar to that of \autoref{lema_45_50}.
\end{proof}

\begin{corollary} \label{convolSum-theor-w_18_36}
Let $n$ be a positive integer. Then 
\begin{align}
 W_{(1,9}(n)  = & 
 \frac{1}{216}\,\sigma_{3}(n)
+ \frac{1}{27}\,\sigma_{3}(\frac{n}{3}) 
+ \frac{3}{8}\,\sigma_{3}(\frac{n}{9}) 
 - \frac{1}{54}\,\EuFrak{b}_{45,1}(n)
\label{convolutionSum-w_1_9}
\end{align}
\begin{align}
 W_{(1,16}(n)  = & 
\frac{1}{768}\,\sigma_{3}(n)     
+\frac{1}{256}\,\sigma_{3}(\frac{n}{2})     
+\frac{1}{64}\,\sigma_{3}(\frac{n}{4})     
+\frac{1}{16}\,\sigma_{3}(\frac{n}{8})   
 +\frac{1}{3}\,\sigma_{3}(\frac{n}{16})      \notag \\ &
    + (\frac{1}{24}-\frac{1}{64}n)\sigma(n) 
   + (\frac{1}{24}-\frac{1}{4}n)\sigma(\frac{n}{16})  
 -\frac{7}{256}\,\EuFrak{b}_{64,1}(n)   \notag \\ &  
 -\frac{3}{64}\,\EuFrak{b}_{64,3}(n)  
-\frac{1}{8} \,\EuFrak{b}'_{64,3}(n)
\label{convolutionSum-w_1_16}
\end{align}
\begin{align}
 W_{(1,18}(n)  = & 
\frac{1}{1080}\,\sigma_{3}(n)     
+\frac{1}{270}\,\sigma_{3}(\frac{n}{2})     
+\frac{1}{135}\,\sigma_{3}(\frac{n}{3})     
+\frac{4}{135}\,\sigma_{3}(\frac{n}{6})   
 +\frac{3}{40}\,\sigma_{3}(\frac{n}{9})      \notag \\ &
  +\frac{3}{10}\,\sigma_{3}(\frac{n}{18})    
    + (\frac{1}{24}-\frac{1}{72}n)\sigma(n) 
   + (\frac{1}{24}-\frac{1}{4}n)\sigma(\frac{n}{18})  
 - \frac{31}{1080}\,\EuFrak{b}_{36,1}(n)   \notag \\ &  
 - \frac{89}{1080}\,\EuFrak{b}_{36,2}(n)  
- \frac{187}{1080} \,\EuFrak{b}_{36,3}(n)
  - \frac{139}{270}\,\EuFrak{b}'_{36,4}(n)  
 + \frac{23}{45} \,\EuFrak{b}_{36,5}(n) 
\label{convolutionSum-w_1_18}
\end{align}
\begin{align}
 W_{(2,9}(n)  = & 
  \frac{1}{1080}\,\sigma_{3}(n)
 + \frac{1}{270}\,\sigma_{3}(\frac{n}{2})
 + \frac{1}{135}\,\sigma_{3}(\frac{n}{3})     
 + \frac{4}{135}\,\sigma_{3}(\frac{n}{6}) 
  + \frac{3}{40}\,\sigma_{3}(\frac{n}{9})      \notag \\ &
  + \frac{3}{10}\,\sigma_{3}(\frac{n}{18}) 
    + (\frac{1}{24}-\frac{1}{36}n)\sigma(\frac{n}{2}) 
   + (\frac{1}{24}-\frac{1}{8}n)\sigma(\frac{n}{9})  
 - \frac{1}{1080}\,\EuFrak{b}_{36,1}(n)   \notag \\ & 
 + \frac{1}{1080}\,\EuFrak{b}_{36,2}(n)
  - \frac{37}{1080} \,\EuFrak{b}_{36,3}(n)
 + \frac{11}{270}\,\EuFrak{b}'_{36,4}(n)   
  - \frac{7}{45}\,\EuFrak{b}_{36,5}(n) 
\label{convolutionSum-w_2_9}
\end{align}
\begin{align}
 W_{(1,25)}(n)  = & 
 \frac{1}{1560}\,\sigma_{3}(n)   
  + \frac{1}{65}\,\sigma_{3}(\frac{n}{5}) 
 + \frac{125}{312}\,\sigma_{3}(\frac{n}{25})  
    + (\frac{1}{24}-\frac{1}{100}n)\sigma(n) \notag \\ &
   + (\frac{1}{24}-\frac{1}{4}n)\sigma(\frac{n}{25})    
  - \frac{21}{650}\,\EuFrak{b}_{50,1}(n)   
  - \frac{1}{5}\,\EuFrak{b}'_{50,2}(n)  
  - \frac{3}{5}\,\EuFrak{b}_{50,3}(n)     \notag \\ &
   - \,\EuFrak{b}_{50,4}(n)  
  - \frac{21}{26}\,\EuFrak{b}_{50,5}(n)    
\label{convolutionSum-w_1_25}
\end{align}
\begin{align}
 W_{(1,36}(n)  = & 
    \frac{1}{4320}\,\sigma_{3}(n) 
   + \frac{7}{1440}\,\sigma_{3}(\frac{n}{2}) 
   + \frac{1}{540}\,\sigma_{3}(\frac{n}{3}) 
    - \frac{17}{270}\,\sigma_{3}(\frac{n}{4}) 
   + \frac{1}{720}\,\sigma_{3}(\frac{n}{6})    \notag \\ &
   + \frac{3}{160}\,\sigma_{3}(\frac{n}{9}) 
    + \frac{13}{135}\,\sigma_{3}(\frac{n}{12}) 
   + \frac{9}{160}\,\sigma_{3}(\frac{n}{18})   
    + \frac{3}{10}\,\sigma_{3}(\frac{n}{36})     \notag \\ &
    + (\frac{1}{24}-\frac{1}{144}n)\sigma(n) 
   + (\frac{1}{24}-\frac{1}{4}n)\sigma(\frac{n}{36})  
    - \frac{151}{4320}\,\EuFrak{b}_{36,1}(n)     
    - \frac{13}{144}\,\EuFrak{b}_{36,2}(n)    \notag \\ &
    - \frac{11}{120}\,\EuFrak{b}_{36,3}(n)
    - \frac{809}{2160}\,\EuFrak{b}_{36,4}(n)
    - \frac{17}{240}\,\EuFrak{b}_{36,5}(n)
    - \frac{19}{48}\,\EuFrak{b}_{36,6}(n)      
    - \frac{139}{360}\,\EuFrak{b}_{36,7}(n)   \notag \\ &
    - \frac{49}{120}\,\EuFrak{b}_{36,8}(n)
    + \frac{5}{16}\,\EuFrak{b}_{36,9}(n)
    - \frac{59}{135}\,\EuFrak{b}_{36,10}(n)
    - \frac{1}{8}\,\EuFrak{b}_{36,11}(n)
     - \frac{1}{16} \,\EuFrak{b}_{36,12}(n)
\label{convolutionSum-w_1_36}
\end{align}
\begin{align}
 W_{(4,9}(n)  = & 
   \frac{1}{4320}\,\sigma_{3}(n)         
  + \frac{7}{1440}\,\sigma_{3}(\frac{n}{2}) 
  + \frac{1}{540}\,\sigma_{3}(\frac{n}{3})  
  - \frac{17}{270}\,\sigma_{3}(\frac{n}{4})  
  - \frac{281}{720}\,\sigma_{3}(\frac{n}{6})     \notag \\ &
  + \frac{3}{160}\,\sigma_{3}(\frac{n}{9})  
  +  \frac{4757}{1080}\,\sigma_{3}(\frac{n}{12})   
  +  \frac{1171}{96}\,\sigma_{3}(\frac{n}{18})    
  - \frac{15991}{120}\,\sigma_{3}(\frac{n}{36})    \notag \\ & 
    + (\frac{1}{24}-\frac{1}{36}n)\sigma(\frac{n}{4}) 
   + (\frac{1}{24}-\frac{1}{16}n)\sigma(\frac{n}{9})   
   - \frac{151}{4320}\,\EuFrak{b}_{36,1}(n)
   - \frac{13}{144}\,\EuFrak{b}_{36,2}(n)    \notag \\ &
   - \frac{11}{120}\,\EuFrak{b}_{36,3}(n)
   - \frac{809}{2160}\,\EuFrak{b}_{36,4}(n)
   - \frac{17}{240}\,\EuFrak{b}_{36,5}(n)   
   - \frac{1}{240}\,\EuFrak{b}_{36,6}(n)
   - \frac{139}{360}\,\EuFrak{b}_{36,7}(n)       \notag \\ &
   - \frac{49}{120}\,\EuFrak{b}_{36,8}(n)
   + \frac{5}{16}\,\EuFrak{b}_{36,9}(n)
   - \frac{59}{135}\,\EuFrak{b}_{36,10}(n)    
   - \frac{1}{8}\,\EuFrak{b}_{36,11}(n)
   - \frac{1}{16}\,\EuFrak{b}_{36,12}(n)  
\label{convolutionSum-w_4_9}
\end{align}
\end{corollary}
\begin{proof} Similar to that of \autoref{convolSum-theor-w_45_50}.
\end{proof}


\section{Concluding Remark} \label{conclusion}

The set of natural number $\mathbb{N}$ can be expressed as the disjoint 
union of the sets $\N$ and $\mathbb{N}\setminus\N$.
When assuming that a basis of the space of cusp forms is determined, 
\eN\ \cite{ntienjem2016c} has evaluated convolution sums for natural 
numbers which belong to $\N$. In this paper we have evaluated convolution 
sums for natural numbers which are in $\mathbb{N}\setminus\N$ making the 
same assumption. 
When we put altogether, we can say that for all natural numbers $\alpha$ 
and $\beta$, the convolution sums for $\alpha\beta$ are evaluated. 

The determination of a basis of the space of cusp forms is tedious, 
especially when $\alpha\beta$ is large and has a large number of divisors. 
An effective and efficient approach to build a basis of the space of 
cusp forms of weight $4$ for $\Gamma_{0}(\alpha\beta)$ is a work in progress.





\begin{thebibliography}{10}
\bibitem{glaisher}
J.~W.~L. Glaisher.
\newblock On the square of the series in which the coefficients are the sums of
  the divisors of the exponents.
\newblock {\em Messenger Math}, 14:156--163, 1862.
\end{thebibliography}



 \section*{Figures}

   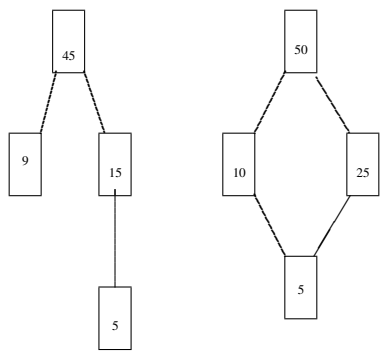
\begin{figure}[ht!]
   \caption{Inclusion relationship of the modular space of weight $4$ for 
   $\Gamma_{0}(45)$ and $\Gamma_{0}(50)$}\label{bas-s-45-50} 
~~\\ 
 
\setlength{\unitlength}{4144sp}%
\begingroup\makeatletter\ifx\SetFigFont\undefined%
\gdef\SetFigFont#1#2#3#4#5{%
  \reset@font\fontsize{#1}{#2pt}%
  \fontfamily{#3}\fontseries{#4}\fontshape{#5}%
  \selectfont}%
\fi\endgroup%
\begin{picture}(2229,2049)(1159,-2098)
\thinlines
{\put(1431,-429){\framebox(185,368){}}
}%
{\put(1171,-1165){\framebox(185,368){}}
}%
{\put(1708,-2086){\framebox(186,368){}}
}%
{\put(1708,-1165){\framebox(186,368){}}
}%
{\put(2820,-429){\framebox(185,368){}}
}%
{\put(2449,-1165){\framebox(185,368){}}
}%
{\put(3191,-1165){\framebox(185,368){}}
}%
{\put(2820,-1902){\framebox(185,368){}}
}%
{\multiput(1449,-429)(-1.95870,-7.83479){48}{\makebox(1.5875,11.1125){\tiny.}}
}%
{\multiput(1616,-429)(2.53125,-7.59375){49}{\makebox(1.5875,11.1125){\tiny.}}
}%
{\put(1801,-1165){\line( 0, 1){ 36}}
\put(1801,-1129){\line( 0,-1){ 36}}
\put(1801,-1165){\line( 0,-1){ 37}}
\put(1801,-1202){\line( 0,-1){ 37}}
\put(1801,-1239){\line( 0,-1){ 37}}
\put(1801,-1276){\line( 0,-1){ 37}}
\put(1801,-1313){\line( 0,-1){ 37}}
\put(1801,-1350){\line( 0,-1){ 37}}
\put(1801,-1387){\line( 0,-1){ 36}}
\put(1801,-1423){\line( 0,-1){ 37}}
\put(1801,-1460){\line( 0,-1){ 37}}
\put(1801,-1497){\line( 0,-1){ 37}}
\put(1801,-1534){\line( 0,-1){ 36}}
\put(1801,-1570){\line( 0,-1){ 37}}
\put(1801,-1607){\line( 0,-1){ 37}}
\put(1801,-1644){\line( 0,-1){ 37}}
\put(1801,-1681){\line( 0,-1){ 37}}
}%
{\multiput(2820,-429)(-3.61569,-7.23137){52}{\makebox(1.5875,11.1125){\tiny.}}
}%
{\multiput(3005,-466)(4.15457,-6.92428){48}{\makebox(1.5875,11.1125){\tiny.}}
}%
{\put(3209,-1165){\line(-3,-5){216.794}}
}%
{\multiput(2634,-1165)(3.55385,-7.10769){53}{\makebox(1.5875,11.1125){\tiny.}}
}%
\put(1783,-1975){\makebox(0,0)[lb]{\smash{{\SetFigFont{5}{6.0}{\rmdefault}{\mddefault}{\updefault}{$5$}%
}}}}
\put(2894,-1755){\makebox(0,0)[lb]{\smash{{\SetFigFont{5}{6.0}{\rmdefault}{\mddefault}{\updefault}{$5$}%
}}}}
\put(1245,-982){\makebox(0,0)[lb]{\smash{{\SetFigFont{5}{6.0}{\rmdefault}{\mddefault}{\updefault}{$9$}%
}}}}
\put(1764,-1055){\makebox(0,0)[lb]{\smash{{\SetFigFont{5}{6.0}{\rmdefault}{\mddefault}{\updefault}{$15$}%
}}}}
\put(1486,-355){\makebox(0,0)[lb]{\smash{{\SetFigFont{5}{6.0}{\rmdefault}{\mddefault}{\updefault}{$45$}%
}}}}
\put(2875,-319){\makebox(0,0)[lb]{\smash{{\SetFigFont{5}{6.0}{\rmdefault}{\mddefault}{\updefault}{$50$}%
}}}}
\put(2505,-1055){\makebox(0,0)[lb]{\smash{{\SetFigFont{5}{6.0}{\rmdefault}{\mddefault}{\updefault}{$10$}%
}}}}
\put(3246,-1055){\makebox(0,0)[lb]{\smash{{\SetFigFont{5}{6.0}{\rmdefault}{\mddefault}{\updefault}{$25$}%
}}}}
\end{picture}%
 
\end{figure}

\begin{figure}[ht!]
\caption{Inclusion relation of the modular space of weight $4$ for 
$\Gamma_{0}(48)$.}
\label{bas-s-48} 
~~\\

\setlength{\unitlength}{4144sp}%
\begingroup\makeatletter\ifx\SetFigFont\undefined%
\gdef\SetFigFont#1#2#3#4#5{%
  \reset@font\fontsize{#1}{#2pt}%
  \fontfamily{#3}\fontseries{#4}\fontshape{#5}%
  \selectfont}%
\fi\endgroup%
\begin{picture}(1914,2049)(1114,-1648)
\thinlines
{\put(1666,119){\framebox(270,270){}}
}%
{\put(1126,-421){\framebox(270,270){}}
}%
{\put(2206,-421){\framebox(270,270){}}
}%
{\put(1666,-961){\framebox(270,270){}}
}%
{\put(2746,-961){\framebox(270,270){}}
}%
{\put(2746,-1636){\framebox(270,270){}}
}%
{\put(1694,119){\line(-6,-5){308.656}}
}%
{\put(1396,-151){\makebox(1.5875,11.1125){\tiny.}}
}%
{\put(1936,119){\line( 1,-1){270}}
}%
{\put(2206,-449){\line(-6,-5){278.361}}
}%
{\put(1396,-449){\line( 6,-5){278.361}}
}%
{\put(1450,-1284){\makebox(1.5875,11.1125){\tiny.}}
}%
{\put(1450,-1284){\line( 0, 1){ 26}}
}%
{\put(1450,-1258){\makebox(1.5875,11.1125){\tiny.}}
}%
{\put(2476,-421){\line( 1,-1){270}}
}%
{\put(2882,-961){\line( 0,-1){405}}
}%
\put(1774,227){\makebox(0,0)[lb]{\smash{{\SetFigFont{7}{8.4}{\rmdefault}{\mddefault}{\updefault}{$48$}%
}}}}
\put(1234,-367){\makebox(0,0)[lb]{\smash{{\SetFigFont{7}{8.4}{\rmdefault}{\mddefault}{\updefault}{$16$}%
}}}}
\put(2288,-340){\makebox(0,0)[lb]{\smash{{\SetFigFont{7}{8.4}{\rmdefault}{\mddefault}{\updefault}{$24$}%
}}}}
\put(1774,-880){\makebox(0,0)[lb]{\smash{{\SetFigFont{7}{8.4}{\rmdefault}{\mddefault}{\updefault}{$8$}%
}}}}
\put(2828,-853){\makebox(0,0)[lb]{\smash{{\SetFigFont{7}{8.4}{\rmdefault}{\mddefault}{\updefault}{$12$}%
}}}}
\put(2854,-1528){\makebox(0,0)[lb]{\smash{{\SetFigFont{7}{8.4}{\rmdefault}{\mddefault}{\updefault}{$6$}%
}}}}
\end{picture}
\end{figure}
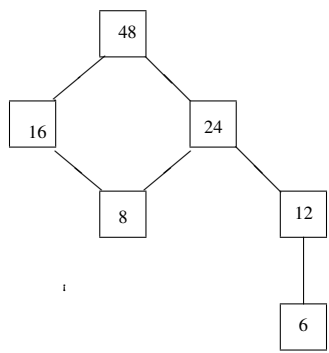




\section*{Tables}

\begin{longtable}{|c|cccccc|} \hline
   & \textbf{1}  &  \textbf{3}  & \textbf{5} & \textbf{9} & \textbf{15}  &
   \textbf{45}  \\ \hline  
\textbf{1}  &  0 & 8 & 0 & 0 & 0 & 0 \\
\textbf{2}  &  2 & 2 & 2 & 0 & 2 & 0 \\
\textbf{3}  &  0 & 4 & 0 & 0 & 4 & 0 \\
\textbf{4}  &  3 & 4 & 0 & -1 & 0 & 2 \\
\textbf{5}  &  2 & 0 & 2 & 2 & 0 & 2 \\
\textbf{6}  &  4 & 0 & 1 & 0 & 0 & 3 \\
\textbf{7}  &  1 & 0 & 1 & 3 & 0 & 3 \\
\textbf{8}  &  3 & 0 & 0 & 1 & 0 & 4 \\
\textbf{9}  &  5 & 0 & -1 & -1 & 0 & 5 \\
\textbf{10}  &  0 & 3 & 0 & -1 & 1 & 5  \\
\textbf{11}  &  1 & 1 & 1 & 0 & -1 & 6  \\
\textbf{12}  &  4 & 0 & 4 & 0 & 0 & 0   \\
\textbf{13}  &  2 & 0 & 2 & 0 & 4 & 0   \\
\textbf{14}  &  0 & -1 & 3 & 9 & -3 & 0   \\ \hline
\caption{Power of $\eta$-quotients being basis elements for $S_{4}(\Gamma_{0}(45))$}
\label{convolutionSums-5_9-table}
\end{longtable}

\begin{longtable}{|c|cccccccccc|} \hline
   & \textbf{1}  &  \textbf{2}  & \textbf{3} & \textbf{4} & \textbf{6} & \textbf{8} & 
   \textbf{12} & \textbf{16} & \textbf{24} & \textbf{48}  \\ \hline
\textbf{1}  &   0  &  4 & 0 & 4 & 0 & 0 & 0 & 0 & 0 & 0  \\ \hline
\textbf{2}  &   0 & 0 & 0 & 4 & 0 & 4 & 0 & 0 & 0 & 0  \\ \hline 
\textbf{3}  &   0 & 0 & 0 & 0 & 4 & 0 & 4 & 0 & 0 & 0  \\ \hline
\textbf{4}  &   0 & 0 & 0 & 2 & 0 & 2 & 2 & 0 & 2 & 0  \\ \hline
\textbf{5}  &   0 & 2 & 0 & -2 & -2 & 2 & 6 & 0 & 2 & 0  \\ \hline
\textbf{6}  &   0 & 0 & 0 & 0 & 0 & 0 & 4 & 0 & 4 & 0  \\ \hline
\textbf{7}  &   0 & 2 & 0 & 0 & -2 & -2 & 4 & 0 & 6 & 0  \\ \hline
\textbf{8}  &   0 & 0 & 0 & 0 & 0 & 2 & 0 & 2 & 2 & 2  \\ \hline
\textbf{9}  &   0 & 0 & 0 & 0 & 0 & 0 & 6 & 0 & -2 & 4  \\ \hline
\textbf{10}  &   0 & 0 & 0 & 0 & 0 & 1 & 4 & 1 & -3 & 5  \\ \hline
\textbf{11}  &   0 & 0 & 0 & 0 & 0 & 3 & 2 & -3 & 1 & 5  \\ \hline
\textbf{12}  &   0 & 0 & 0 & 0 & 0 & 4 & 0 & -2 & 0 & 6  \\ \hline
\textbf{13}  &   0 & 0 & 0 & -1 & 0 & 7 & 1 & -4 & -3 & 8  \\ \hline
\textbf{14}  &   0 & 0 & 0 & 0 & 0 & 3 & 4 & -3 & -5 & 9  \\ \hline
\textbf{15}  &   0 & 0 & 0 & 0 & 0 & 4 & 2 & -2 & -6 & 10  \\ \hline 
\textbf{16}  &   0 & -1 & 0 & -1 & 3 & 3 & -1 & -3 & -1 & 9  \\ \hline
\textbf{17}  &   0 & 0 & 0 & 2 & 0 & 0 & 0 & -2 & -2 & 10  \\ \hline
\textbf{18}  &   3 & 0 & 1 & 0 & -1 & 0 & 0 & 3 & 5 & -3  \\ \hline
\caption{Power of $\eta$-functions being basis elements of $\S_{4}(\Gamma_{0}(48))$}
\label{convolutionSums-3_16-table}
\end{longtable}

\begin{longtable}{|c|cccccc|} \hline
   & \textbf{1}  &  \textbf{2}  & \textbf{5} & \textbf{10} & \textbf{25}  &
   \textbf{50}  \\  \hline
\textbf{1}  & 4 & 0 & 4 & 0 & 0 & 0 \\
\textbf{2}  & 0 & 4 & 0 & 4 & 0 & 0 \\
\textbf{3}  & 2 & 0 & 4 & 0 & 2 & 0 \\
\textbf{4}  & 1 & 0 & 4 & 0 & 3 & 0 \\
\textbf{5}  & 0 & 0 & 4 & 0 & 4 & 0 \\
\textbf{6}  & 0 & 2 & 0 & 4 & 0 & 2 \\
\textbf{7}  & 0 & 4 & 2 & 0 & -2 & 4 \\
\textbf{8}  & 0 & 1 & 0 & 4 & 0 & 3  \\
\textbf{9}  & 1 & 0 & 0 & 4 & -1 & 4 \\
\textbf{10}  & 0 & 0 & 0 & 4 & 0 & 4 \\
\textbf{11}  & 0 & 2 & 2 & 0 & -2 & 6 \\
\textbf{12}  & 0 & -1 & 0 & 4 & 0 & 5 \\
\textbf{13}  & 0 & 1 & 2 & 0 & -2 & 7 \\
\textbf{14}  & 1 & 0 & 2 & 0 & -3 & 8 \\
\textbf{15}  & 0 & 0 & 2 & 0 & -2 & 8 \\
\textbf{16}  & -1 & 0 & 6 & -2 & -5 & 10 \\
\textbf{17}  & 0 & -1 & 1 & 3 & -5 & 10 \\ \hline
\caption{Power of $\eta$-quotients being basis elements for $S_{4}(\Gamma_{0}(50))$}
\label{convolutionSums-2_25-table}
\end{longtable}

\begin{longtable}{|c|ccccccc|} \hline
   & \textbf{1}  &  \textbf{2}  & \textbf{4} & \textbf{8} & \textbf{16} & \textbf{32}  
   & \textbf{64} \\ \hline
\textbf{1}  & 0 & 4 & 4 & 0 & 0 & 0 & 0  \\ \hline
\textbf{2}  & 0 & 0 & 4 & 4 & 0 & 0 & 0  \\ \hline
\textbf{3}  & 0 & 4 & 0 & 0 & 4 & 0 & 0  \\ \hline
\textbf{4}  & 0 & 0 & 0 & 4 & 4 & 0 & 0  \\ \hline
\textbf{5}  & 0 & 2 & 1 & 0 & 3 & 2 & 0  \\ \hline
\textbf{6}  & 0 & 0 & 4 & 0 & 0 & 4 & 0  \\ \hline
\textbf{7}  & 0 & 0 & 2 & 0 & 2 & 4 & 0  \\ \hline
\textbf{8}  & 0 & 0 & 0 & 0 & 4 & 4 & 0 \\ \hline
\textbf{9}  & 0 & 0 & 2 & 2 & -4 & 8 & 0  \\ \hline
\textbf{10}  & 0 & 0 & 0 & 2 & -2 & 8 & 0  \\ \hline
\textbf{11}  & 0 & 0 & 0 & 1 & 2 & 3 & 2  \\ \hline
\textbf{12}  & 0 & 0 & 0 & 0 & 6 & -2 & 4  \\ \hline
\textbf{13}  & 0 & -4 & 10 & -1 & 0 & -3 & 6  \\ \hline
\textbf{14}  & 0 & 0 & 0 & 2 & 0 & 2 & 4 \\ \hline
\textbf{15}  & 0 & 0 & 0 & 1 & 4 & -3 & 6  \\ \hline
\textbf{16}  & 0 & 0 & -4 & 6 & 2 & 4 & 0  \\ \hline
\textbf{17}  & 0 & 0 & -2 & 8 & -2 & -4 & 8  \\ \hline
\textbf{18}  & 0 & 0 & 0 & 2 & 2 & -4 & 8  \\ \hline
\caption{Power of $\eta$-functions being basis elements of $\S_{4}(\Gamma_{0}(64))$}
\label{convolutionSums-1_64-table}
\end{longtable}

\begin{longtable}{|c|ccccccccc|} \hline
   & \textbf{1}  &  \textbf{2}  & \textbf{3} & \textbf{4} & \textbf{6} & \textbf{9}  
   & \textbf{12} & \textbf{18} & \textbf{36}  \\ \hline
\textbf{1}  &  0 & 0 & 8 & 0 & 0 & 0 & 0 & 0 & 0 \\ \hline
\textbf{2}  &  0 & 0 & 0 & 0 & 8 & 0 & 0 & 0 & 0 \\ \hline
\textbf{3}  &  0 & 0 & 2 & 0 & 2 & 2 & 0 & 2 & 0 \\ \hline
\textbf{4}  &  0 & 0 & 3 & 0 & 1 & -1 & 0 & 5 & 0 \\ \hline
\textbf{5}  &  0 & 0 & 4 & 0 & 0 & -4 & 0 & 8 & 0 \\ \hline
\textbf{6}  &  0 & 0 & 0 & 0 & 2 & 0 & 2 & 2 & 2 \\ \hline
\textbf{7}  &  0 & 0 & 0 & 0 & 3 & 0 & -1 & 3 & 3 \\ \hline
\textbf{8}  &  0 & 0 & 0 & 0 & 3 & 0 & 1 & -1 & 5 \\ \hline
\textbf{9}  &  0 & 0 & 0 & 0 & 4 & 0 & -2 & 0 & 6 \\ \hline
\textbf{10}  &  0 & 0 & 0 & 0 & 4 & 0 & 0 & -4 & 8 \\ \hline
\textbf{11}  &  0 & 0 & 0 & 0 & 5 & 0 & -3 & -3 & 9 \\ \hline
\textbf{12}  &  -5 & 11 & 5 & -5 & -1 & -2 & 0 & 0 & 5 \\ \hline
\caption{Power of $\eta$-functions being basis elements of $\S_{4}(\Gamma_{0}(36))$}
\label{convolutionSums-1_36-table}
\end{longtable}


\end{document}